\documentclass[11pt,reqno]{amsart}
\pdfoutput=1 

\usepackage[english]{babel}
\usepackage[utf8]{inputenc}
\usepackage[T1]{fontenc}
\usepackage{lmodern} \normalfont 
\DeclareFontShape{T1}{lmr}{bx}{sc} { <-> ssub * cmr/bx/sc }{}
\usepackage{microtype}	

\usepackage{amssymb}
\usepackage{amsmath}
\usepackage{amsthm}
\usepackage{mathtools}
\usepackage{mathrsfs}
\usepackage{accents} 
\usepackage{etoolbox}
\usepackage{dsfont} 

\usepackage{graphicx}
\usepackage{color}
\usepackage[dvipsnames]{xcolor}
\usepackage{tikz}
\usetikzlibrary{calc,positioning,shapes}
\usepackage{pgfplots}
\pgfplotsset{colormap name=viridis,compat=newest}
\usepackage[margin=10pt,font=small,labelfont=bf,labelsep=endash]{caption}
\usepackage{subcaption}

\usepackage{booktabs}
\usepackage{paralist}

\numberwithin{equation}{section}

\usepackage{enumitem} 
\setlist[enumerate]{label=(\roman*)}
\usepackage{xspace}

\usepackage[colorlinks,linkcolor=teal,citecolor=teal,urlcolor=teal]{hyperref}
\usepackage[nameinlink,noabbrev]{cleveref}
\AddToHook{cmd/appendix/before}{%
    \crefalias{section}{appendix}%
    \crefalias{subsection}{appendix}
}

\theoremstyle{plain}
\newtheorem{theorem}{Theorem}[section]
\newtheorem{proposition}[theorem]{Proposition}
\AddToHook{env/proposition/begin}{\crefalias{theorem}{proposition}}
\Crefname{proposition}{Proposition}{Propositions}

\newtheorem{lemma}[theorem]{Lemma}
\AddToHook{env/lemma/begin}{\crefalias{theorem}{lemma}}
\Crefname{lemma}{Lemma}{Lemmata}

\newtheorem{corollary}[theorem]{Corollary}
\AddToHook{env/corollary/begin}{\crefalias{theorem}{corollary}}
\Crefname{corollary}{Corollary}{Corollaries}

\newtheorem{remark}[theorem]{Remark}
\AddToHook{env/remark/begin}{\crefalias{theorem}{remark}}
\Crefname{remark}{Remark}{Remarks}

\AddToHook{env/definition/begin}{\crefalias{theorem}{definition}}
\Crefname{definition}{Definition}{Definitions}

\newtheorem{assumption}[theorem]{Assumption}
\AddToHook{env/assumption/begin}{\crefalias{theorem}{assumption}}
\Crefname{assumption}{Assumption}{Assumptions}

\newtheorem{example}[theorem]{Example}
\AddToHook{env/example/begin}{\crefalias{theorem}{example}}
\Crefname{example}{Example}{Examples}

\newcommand{\R}{\ensuremath\mathbb{R}}
\newcommand{\C}{\ensuremath\mathbb{C}}

\newcommand{\T}{\ensuremath\mathsf{T}}

\newcommand{\dtheta}{\,\mathrm{d}\theta}
\newcommand{\eul}{\mathrm{e}}

\newcommand{\hook}{\ensuremath{\hookrightarrow}}

\DeclareMathOperator{\id}{id}

\DeclareMathOperator{\diag}{diag}

\DeclareMathOperator{\real}{Re}
\DeclareMathOperator{\imag}{Im}
\DeclareMathOperator{\tol}{TOL}
\DeclareMathOperator{\tr}{tr}

\newcommand{\vertiii}[1]{{\left\vert\kern-0.25ex\left\vert\kern-0.25ex\left\vert #1 \right\vert\kern-0.25ex\right\vert\kern-0.25ex\right\vert}}

\newcommand{\cA}{\ensuremath{\mathcal{A}}}
\newcommand{\cB}{\ensuremath{\mathcal{B}}}
\newcommand{\cC}{\ensuremath{\mathcal{C}}}
\newcommand{\cD}{\ensuremath{\mathcal{D}}}
\newcommand{\cE}{\ensuremath{\mathcal{E}}}
\newcommand{\cF}{\ensuremath{\mathcal{F}}}

\newcommand{\cH}{\ensuremath{\mathcal{H}}}

\newcommand{\cK}{\ensuremath{\mathcal{K}}}
\newcommand{\cL}{\ensuremath{\mathcal{L}}}
\newcommand{\cM}{\ensuremath{\mathcal{M}}}

\newcommand{\cO}{\ensuremath{\mathcal{O}}}

\newcommand{\cQ}{\ensuremath{\mathcal{Q}}}
\newcommand{\cR}{\ensuremath{\mathcal{R}}}

\newcommand{\cT}{\ensuremath{\mathcal{T}}}

\newcommand{\cV}{\ensuremath{\mathcal{V}}}

\newcommand{\cX}{\ensuremath{\mathcal{X}}}
\newcommand{\cY}{\ensuremath{\mathcal{Y}}}

\newcommand{\cHV}{\ensuremath{\cH_{\scalebox{.5}{\cV}}}}
\newcommand{\cHQ}{\ensuremath{\cH_{\scalebox{.5}{\cQ}}}}

\newcommand{\A}{\ensuremath\mathbb{A}} 

\newcommand{\lineWidth}{1.2pt}

\definecolor{color0}{rgb}{1.0, 0.0, 0.0}
\definecolor{color1}{rgb}{0.0, 0.0, 1.0}
\definecolor{color2}{rgb}{0.1, 0.2, 0.9}
\definecolor{color3}{rgb}{0.9, 0.2, 0.1}
\definecolor{color4}{rgb}{0.8, 0.2, 0.8}
\definecolor{color5}{rgb}{0,0,0}

\definecolor{rabred}{rgb}{1.0, 0.0, 0.0}
\definecolor{rabblue}{rgb}{0.0, 0.0, 1.0}
\definecolor{rabcol}{rgb}{0.1, 0.2, 0.9}

\definecolor{crimson2143940}{RGB}{214,39,40}
\definecolor{darkgray176}{RGB}{176,176,176}
\definecolor{darkorange25512714}{RGB}{255,127,14}
\definecolor{forestgreen4416044}{RGB}{44,160,44}
\definecolor{gray127}{RGB}{127,127,127}
\definecolor{mediumpurple148103189}{RGB}{148,103,189}
\definecolor{orchid227119194}{RGB}{227,119,194}
\definecolor{sienna1408675}{RGB}{140,86,75}
\definecolor{steelblue31119180}{RGB}{31,119,180}

\definecolor{cbsblue}{RGB}{68,119,170}
\definecolor{cbscyan}{RGB}{102,204,238}
\definecolor{cbsgreen}{RGB}{34,136,51}
\definecolor{cbsyellow}{RGB}{204,187,68}
\definecolor{cbsred}{RGB}{238,102,119}
\definecolor{cbspurple}{RGB}{170,51,119}
\definecolor{cbsgrey}{RGB}{187,187,187}

\definecolor{mycolor1}{rgb}{0.00000,0.44700,0.74100}
\definecolor{mycolor2}{rgb}{0.85000,0.32500,0.09800}
\definecolor{mycolor3}{rgb}{0.92900,0.69400,0.12500}
\definecolor{mycolor4}{rgb}{0.46600,0.67400,0.18800}
\definecolor{mycolor5}{rgb}{0.49400,0.18400,0.55600}
\definecolor{mycolor6}{rgb}{0.16000,0.49000,0.49000}%
\definecolor{mycolor7}{rgb}{0.10000,0.49000,0.80000}%
\definecolor{mycolor8}{rgb}{0.06640,0.46484,0.19921} 
\definecolor{mycolor9}{rgb}{0.59765,0.59765,0.19921} 

\pgfplotscreateplotcyclelist{summationsol}{%
line width = \lineWidth, color=mycolor1\\
line width = \lineWidth, color=mycolor2\\
line width = \lineWidth, color=mycolor3\\
line width = \lineWidth, color=mycolor4\\
line width = \lineWidth, color=mycolor5\\
line width = \lineWidth, color=mycolor6\\
line width = \lineWidth, color=mycolor9\\
}

\pgfplotscreateplotcyclelist{toomuchdata}{%
semithick, color=forestgreen4416044\\
semithick, color=darkorange25512714\\
semithick, color=gray127\\
semithick, color=mediumpurple148103189\\
semithick, color=orchid227119194\\
semithick, color=sienna1408675\\
semithick, color=steelblue31119180\\
semithick, color=crimson2143940\\
}

\pgfplotscreateplotcyclelist{errresnorm}{%
thick, color=forestgreen4416044\\
thick, color=darkorange25512714\\
thick, color=gray127\\
semithick, color=forestgreen4416044\\
semithick, color=darkorange25512714\\
semithick, color=gray127\\
}

\pgfplotscreateplotcyclelist{rab}{%
dotted, color=color0, every mark/.append style={scale=2, solid, fill=color0}, mark=o\\
dashdotted, color=color0, every mark/.append style={scale=2, solid, fill=color1}, mark=o\\
solid, color=color0, every mark/.append style={scale=2, solid, fill=gray},mark=o\\
dotted, color=color1, every mark/.append style={scale=2, solid, fill=color0},mark=triangle\\
dashdotted, color=color1, every mark/.append style={scale=2, solid, fill=color1},mark=triangle\\
solid, color=color1, every mark/.append style={scale=2, solid, fill=gray},mark=triangle\\
dotted, every mark/.append style={solid, fill=gray},mark=*\\
dashdotted, every mark/.append style={solid, fill=gray},mark=square*\\
solid, every mark/.append style={solid, fill=gray}, mark=triangle*\\
}

\pgfplotscreateplotcyclelist{methodcompare}{%
color=cbsblue, every mark/.append style={scale=1.8, solid, fill=cbsblue}, very thick, mark=o\\
color=cbsblue, every mark/.append style={scale=2.2, solid, fill=cbsblue}, very thick, mark=diamond\\
color=cbsblue, every mark/.append style={scale=2, solid, fill=cbsblue}, very thick, mark=triangle\\
densely dashed, color=cbsred, every mark/.append style={scale=1.8, solid, fill=cbsred}, very thick, mark=o\\
densely dashed, color=cbsred, every mark/.append style={scale=2.2, solid, fill=cbsred}, very thick, mark=diamond\\
densely dashed , color=cbsred, every mark/.append style={scale=2, solid, fill=cbsred}, very thick, mark=triangle\\
loosely dotted, every mark/.append style={solid, fill=gray}, mark=*\\
loosely dotted, every mark/.append style={scale=1.2, solid, fill=gray}, mark=diamond*\\
loosely dotted, every mark/.append style={scale=1.2, solid, fill=gray}, mark=triangle*\\
}

\pgfplotscreateplotcyclelist{rkconv}{%
densely dashed, color=cbsred, every mark/.append style={scale=1.8, solid, fill=cbsred}, very thick, mark=o\\
densely dashed, color=cbsred, every mark/.append style={scale=2.2, solid, fill=cbsred}, very thick, mark=triangle\\
densely dashed, color=cbsred, every mark/.append style={scale=2.0, solid, fill=cbsred}, very thick, mark=pentagon\\
color=cbsblue, every mark/.append style={scale=1.8, solid, fill=cbsblue}, very thick, mark=o\\
color=cbsblue, every mark/.append style={scale=2.2, solid, fill=cbsblue}, very thick, mark=triangle\\
color=cbsblue, every mark/.append style={scale=2.0, solid, fill=cbsblue}, very thick, mark=pentagon\\
loosely dashed, every mark/.append style={solid, fill=gray}, mark=*\\
loosely dotted, every mark/.append style={scale=1.2, solid, fill=gray}, mark=triangle*\\
dash dot, every mark/.append style={scale=1.2, solid, fill=gray}, mark=pentagon*\\
}

\pgfplotscreateplotcyclelist{rkconvIter}{%
densely dashed, color=cbsred, every mark/.append style={scale=1.8, solid, fill=cbsred}, very thick, mark=o\\
densely dashed, color=cbsred, every mark/.append style={scale=2.2, solid, fill=cbsred}, very thick, mark=triangle\\
densely dashed, color=cbsred, every mark/.append style={scale=2.0, solid, fill=cbsred}, very thick, mark=pentagon\\
color=cbsblue, every mark/.append style={scale=1.8, solid, fill=cbsblue}, very thick, mark=o\\
color=cbsblue, every mark/.append style={scale=2.2, solid, fill=cbsblue}, very thick, mark=triangle\\
color=cbsblue, every mark/.append style={scale=2.0, solid, fill=cbsblue}, very thick, mark=pentagon\\
color=cbsgreen, every mark/.append style={scale=1.8, solid, fill=cbsgreen}, very thick, mark=o\\
color=cbsgreen, every mark/.append style={scale=2.2, solid, fill=cbsgreen}, very thick, mark=triangle\\
color=cbsgreen, every mark/.append style={scale=2.0, solid, fill=cbsgreen}, very thick, mark=pentagon\\
loosely dashed, every mark/.append style={solid, fill=gray}, mark=*\\
loosely dotted, every mark/.append style={scale=1.2, solid, fill=gray}, mark=triangle*\\
dash dot, every mark/.append style={scale=1.2, solid, fill=gray}, mark=pentagon*\\
}

\pgfplotscreateplotcyclelist{iterativeconv}{%
color=mycolor1, every mark/.append style={scale=1.8, solid, fill=mycolor1}, thick, mark=o\\
color=mycolor1, every mark/.append style={scale=2.2, solid, fill=mycolor1}, thick, mark=diamond\\
color=mycolor1, every mark/.append style={scale=2, solid, fill=mycolor1}, thick, mark=triangle\\
loosely dashed, color=mycolor4, every mark/.append style={scale=1.8, solid, fill=mycolor4}, thick, mark=o\\
densely dashed, color=mycolor4, every mark/.append style={scale=2.2, solid, fill=mycolor4}, thick, mark=diamond\\
loosely dotted, color=mycolor4, every mark/.append style={scale=2, solid, fill=mycolor4}, thick, mark=triangle\\
loosely dashed, every mark/.append style={solid, fill=gray}, mark=*\\
densely dashed, every mark/.append style={scale=1.2, solid, fill=gray}, mark=diamond*\\
loosely dotted, every mark/.append style={scale=1.2, solid, fill=gray}, mark=triangle*\\
densely dotted, every mark/.append style={solid, fill=gray}, mark=square*\\
densely dashdotted,every mark/.append style={solid, fill=gray}, mark=pentagon*\\
}

\pgfplotscreateplotcyclelist{itertol}{%
loosely dashed, color=crimson2143940, every mark/.append style={scale=1.6, solid, fill=crimson2143940}, thick, mark=*\\
densely dashed, color=darkorange25512714, every mark/.append style={scale=2.0, solid, fill=darkorange25512714}, thick, mark=diamond*\\
loosely dotted, color=steelblue31119180, every mark/.append style={scale=1.9, solid, fill=steelblue31119180}, thick, mark=triangle*\\
densely dotted, color=forestgreen4416044, every mark/.append style={scale=1.3, solid, fill=forestgreen4416044}, thick, mark=square*\\
every mark/.append style={solid, fill=gray}\\
densely dashed, every mark/.append style={solid, fill=gray}, mark=triangle*
loosely dotted, every mark/.append style={solid, fill=gray}, mark=*\\%
densely dotted, every mark/.append style={solid, fill=gray}, mark=square*\\%
densely dashdotted,every mark/.append style={solid, fill=gray}, mark=pentagon*\\%
}

\newcommand{\abbr}[1]{#1\xspace}
\newcommand{\BDF}{\abbr{BDF}}

\newcommand{\RK}{\abbr{RK}}

\newcommand{\smex}{\sigma} 

\newcommand{\RKA}{\A}
\newcommand{\RKb}{\mathbf{\beta}}
\newcommand{\RKc}{\mathbf{\mathbb{\chi}}}
\newcommand{\abstractE}{\cE} 
\newcommand{\abstractF}{\cF} 
\newcommand{\abstractG}{h} 

\newcommand{\delayOper}[1]{\Psi_{#1}}
\newcommand{\delayOperK}{\delayOper{k}}

\title[Decoupling Runge--Kutta schemes for elliptic--parabolic problems]{Decoupling Runge--Kutta schemes for\\ elliptic--parabolic problems}
\author{Robert Altmann${}^\dagger$ \and Abdullah Mujahid${}^{\star}$ \and Benjamin Unger${}^\ddagger$}

\address{${}^{\dagger}$ Institute of Analysis and Numerics, Otto von Guericke University Magdeburg, Universit\"atsplatz 2, 39106 Magdeburg, Germany}
\email{robert.altmann@ovgu.de}
\address{${}^{\star}$ Stuttgart Center for Simulation Science (SC SimTech), University of Stuttgart, Universit\"{a}tsstr.~32, 70569 Stuttgart, Germany}
\email{abdullah.mujahid@simtech.uni-stuttgart.de}
\address{${}^{\ddagger}$ Institute for Applied and Numerical Mathematics, Karlsruhe Institute of Technology, Englerstr. 2, 76131 Karlsruhe, Germany}
\email{benjamin.unger@kit.edu}
\date{\today}

\begin{document}

\begin{abstract}
We study the construction and convergence of semi-explicit and
iterative decoupling schemes for an elliptic--parabolic problem using higher-order Runge--Kutta methods. 
For the semi-explicit schemes, which are constructed using a nearby delay system with $k$ time delays, we establish the convergence of $k$th-order Runge--Kutta methods under a weak coupling condition. 
We develop the convergence analysis by adapting the Fourier stability and perturbation techniques of [Lubich, Ostermann, Math.~Comp., 64(210):601--627, 1995]. 
The key tool is the generating function framework,
in which the Runge--Kutta discretization is encoded through an operator-valued function. 
Stability estimates are then obtained via Parseval's identity on the unit circle.
We further present convergence results for iterative
(fixed-stress and undrained-split) higher-order Runge--Kutta schemes. Here, a spectral decomposition of the Schur complement operator is central.
Finally, we provide numerical examples to verify the proven convergence results. 
\end{abstract}

\maketitle
{\footnotesize \textsc{Keywords:} Runge--Kutta methods, Fourier stability, semi-explicit schemes, iterative decoupling} 

{\footnotesize \textsc{AMS subject classification:} 65M12, 65J10}


\section{Introduction}
\label{sec:intro}

This article explores decoupling time integration schemes based on
implicit Runge--Kutta (\RK) methods for linear elliptic--parabolic problems, including the equations of poroelasticity~\cite{Bio41}.
Typical applications involve biomechanics,
where the human brain and heart are modeled as poroelastic
media with multiple fluid networks~\cite{VarCT+16,EliRT23}, as well as geomechanics~\cite{Zob10}.

The well-posedness of the considered elliptic--parabolic problem
is studied in~\cite{Sho00}, its spatial discretization discussed in~\cite{ErnM09}. 
To reduce the computational effort of the coupling,
several decoupling strategies exist, based on fixed-point iterations combined with an implicit Euler discretization in time~\cite{MikW13,KimTJ11a,KimTJ11b}.
In \cite{AltMU24b}, fixed-point iterations for operator splitting were combined with higher-order backward differentiation formulae (\BDF) up to order $5$.
An alternative approach based on semi-explicit methods,
which decouples the system through a time delay approximation,
was introduced for the implicit Euler method
in~\cite{AltMU21b} and extended to higher-order \BDF methods in~\cite{AltMU24,AltMU26}.
Such semi-explicit methods require a certain weak coupling condition, which may be relaxed by the implementation of an additional inner iteration~\cite{AltD24,AltD25}. 

The convergence analysis for \BDF-based
semi-explicit schemes presented in~\cite{AltMU26} relies on the construction of a weighted norm
with a symmetric positive definite matrix enabling a telescoping
argument (G-stability), cf.~\cite{NevO81}.
The present article takes a fundamentally different analytical approach.
Instead of G-stability, we adapt the Fourier stability and
perturbation techniques developed by Lubich and
Ostermann in~\cite{LubO95} for \RK methods
of quasi-linear parabolic equations.
The core idea is to encode the \RK discretization through generating functions and an associated operator-valued resolvent, decompose it via Schur decomposition, and obtain stability estimates through Parseval's identity on the unit circle.

Extending the framework of~\cite{LubO95} from a single parabolic
equation to the coupled elliptic--parabolic system with delayed coupling operators is the main analytical challenge tackled in this paper. 
To be more precise, the coupling introduces a Schur complement operator
and the delay approximation modifies the structure
of the operator which needs to be inverted.
Through a spectral decomposition of the Schur complement
and a Rayleigh quotient argument for the coercive diffusion
operator, we reduce the invertibility analysis to a scalar
condition on the eigenvalues of the Schur complement,
bounded by the coupling strength.
The analysis then reveals the same critical coupling bounds
as in the \BDF setting considered in~\cite{AltMU26}. 

For iterative \RK schemes, we combine the contraction analysis of the
iteration with \RK consistency estimates.
Here again, the spectral decomposition of the Schur complement is crucial to
establish the contraction property. 
\smallskip

To summarize, the main contributions of this paper are:
\begin{enumerate}[itemsep=0.2em]
	\item Convergence analysis for \RK-based semi-explicit
	decoupling schemes using Fourier stability techniques,
	establishing convergence of order~$k$ under a weak
	coupling condition and sufficient spatial regularity.
	\item Unified perspective connecting the \BDF
	analysis of~\cite{AltMU26} with the \RK analysis,
	showing that both approaches lead to equivalent stability conditions. 
	\item Convergence analysis for iterative (fixed-stress and undrained-split)
	\RK decoupling schemes, combining contraction analysis with
	\RK consistency estimates.
\end{enumerate}

The remainder of this article is organized as follows.
After this introduction, the abstract model problem is introduced
in \Cref{sec:prelim} with the particular example of poroelasticity.
The delay approximation, the resulting semi-explicit scheme, and its Fourier stability and convergence analysis
are presented in \Cref{sec:semiexplicit}. 
This is followed by the proof of convergence for iterative \RK schemes in \Cref{sec:iterative}. 
Finally, we present a numerical study of the convergence results in \Cref{sec:numerics}.

\subsection*{Notation}
Throughout the article, we write~$a \lesssim b$ to indicate that there exists a
generic constant~$C > 0$, independent of spatial and temporal discretization
parameters, such that~$a \leq C b$.

\section{Problem Setting and Preliminaries}
\label{sec:prelim}
We consider a linear elliptic--parabolic system in an $m$-dimensional bounded Lipschitz domain~$\Omega \subseteq \R^m$, $m\in\{2,3\}$,
over a time interval~$[0,T]$.
Let
\begin{equation*}
	\cV\vcentcolon=[H_{0}^{1}(\Omega)]^m, \qquad
	\cQ\vcentcolon=H_{0}^{1}(\Omega), \qquad
	\cHV\vcentcolon=[L^{2}(\Omega)]^m, \qquad
	\cHQ\vcentcolon=L^{2}(\Omega)
\end{equation*}
and denote by $\cV \hook \cHV\simeq\cHV^{*}\hook\cV^{*}$
and $\cQ \hook \cHQ\simeq\cHQ^{*}\hook\cQ^{*}$
the associated Gelfand triples~\cite[Sect.~23.4]{Zei90}.
Given source terms $f\colon[0,T] \to \cV^*$
and $g\colon[0,T] \to \cQ^*$ of sufficient regularity,
we seek $u\colon [0,T]\rightarrow\cV$ and $p\colon [0,T]\rightarrow\cQ$
such that
\begin{subequations}
	\label{eq:ellpar}
	\begin{align}
		a(u,v) - d(v, p)
		&= \langle f, v \rangle, \label{eq:ellpar:a} \\
		d(\dot u, q) + c(\dot p,q) + b(p,q)
		&= \langle g, q\rangle \label{eq:ellpar:b}
	\end{align}
	for all test functions~$v\in \cV$, $q \in \cQ$,
	and for almost every~$t \in (0,T]$.
	The initial data
	\begin{align}
		u(0) = u^0 \in \cV, \qquad
		p(0) = p^0 \in \cHQ \label{eq:ellpar:c}
	\end{align}
\end{subequations}
are assumed to be \emph{consistent}, i.e., to satisfy the consistency condition
\begin{align*}
	a(u^0,v) - d(v, p^0) = \langle f(0),v\rangle\qquad\text{ for all~$v\in \cV$}.
\end{align*}
Throughout the manuscript, we assume the bilinear forms~$a\colon \cV\times\cV \to \R$,
$b\colon \cQ\times\cQ\to\R$, and $c\colon \cHQ\times\cHQ\to\R$
to be symmetric, continuous, and elliptic.
We write~$c_{\mathfrak{a}}$ and~$C_{\mathfrak{a}}$
for the ellipticity and continuity constants
of~$\mathfrak{a}\in\{a, b, c\}$, respectively.
The norms
\[
	\|\cdot\|_{b} \vcentcolon= b(\cdot,\cdot)^{1/2}
	\qquad\text{and}\qquad
	\|\cdot\|_{c} \vcentcolon= c(\cdot,\cdot)^{1/2},
\]
induced by the bilinear forms~$b$ and~$c$, satisfy
\begin{align*}
	\tfrac{1}{C_b}\,\Vert \cdot \Vert^{2}_{b}
	\le \Vert \cdot \Vert_{\cQ}^{2}
	\le \tfrac{1}{c_b}\,\Vert \cdot \Vert^{2}_{b}
	\qquad\text{and}\qquad
	\tfrac{1}{C_c}\,\Vert \cdot \Vert^{2}_{c}
	\le \Vert \cdot \Vert_{\cHQ}^{2}
	\le \tfrac{1}{c_c}\,\Vert \cdot \Vert^{2}_{c}.
\end{align*}
The dual norms on $\cQ^*$ and $\cV^*$ are written as
$\Vert \cdot \Vert_{\cQ^*}$ and $\Vert \cdot \Vert_{\cV^*}$, respectively.
The coupling form~$d\colon\cV\times\cHQ\to\R$ is bounded, i.e., there exists a constant~$C_d > 0$ such that~$d(u,p) \leq C_d\, \Vert u \Vert_{\cV} \Vert p\Vert_{\cHQ}$ for all~$u \in \cV$, $p\in \cHQ$. Under these assumptions, well-posedness is established in~\cite{Sho00}.
\begin{example}[linear poroelasticity]
\label{ex:biot}
The quasi-static Biot model~\cite{Bio41} of linear poroelasticity
with homogeneous Dirichlet boundary conditions fits into the
framework~\eqref{eq:ellpar}.
Here, the unknowns are the displacement~$u\colon [0,T]\times\Omega\rightarrow\R^{m}$
and the pore pressure~$p\colon [0,T]\times\Omega\rightarrow\R$, satisfying
\begin{subequations}
	\label{eq:pdes}
	\begin{align}
		- \nabla\cdot\sigma(u) + \alpha \nabla p
		&= {\hat f} \qquad\text{in }(0,T]\times\Omega, \label{eq:pdes:a}\\
		\partial_{t} \big(\alpha \nabla\cdot u + \tfrac{1}{M} p\big)
		- \nabla\cdot(\kappa\nabla p)
		&= {\hat g} \qquad\text{in } (0,T]\times\Omega. \label{eq:pdes:b}
	\end{align}
\end{subequations}
Therein, $\sigma(u) = {\mu}\, \big(\nabla u + (\nabla u)^\T \big)
+ {\lambda}\, (\nabla \cdot u) \id$
is the stress tensor with Lam\'{e} coefficients ${\lambda}$ and ${\mu}$,
$\kappa$ denotes the permeability,
$\alpha$ the Biot--Willis coupling coefficient,
and $M$ the Biot modulus.
The identification of the bilinear forms with the
abstract setting~\eqref{eq:ellpar} is standard;
see, e.g.,~\cite{ErnM09}.
\end{example}
We define the \emph{coupling strength}
\begin{equation}
	\label{eqn:coupling:strength}
	\omega\vcentcolon=\frac{C_d^{2}}{c_a c_c},
\end{equation}
which governs the convergence of all decoupling schemes considered
in this article, and plays a central role in the upcoming analysis.

We further introduce the operators
$\cA\colon\cV\rightarrow\cV^*$, $\cB\colon\cQ\rightarrow\cQ^*$,
$\cC\colon\cHQ\rightarrow\cHQ^{*}$,
and $\cD\colon\cV\rightarrow\cHQ^{*}$
associated with $a$, $b$, $c$, and $d$, respectively.
In operator notation, system~\eqref{eq:ellpar} becomes
\begin{subequations}
\label{eq:ellpar:opt}
    \begin{align}
        \cA u - \cD^{*} p
        &= f \qquad \text{in } \cV^{*}, \label{eq:ellpar:opt:a}\\
        \cD{\dot u} + \cC{\dot p} + \cB p
        &= g \hspace{0.79cm} \text{in } \cQ^{*}.\label{eq:ellpar:opt:b}
    \end{align}
\end{subequations}
Introducing the vectors and operator matrices
\begin{equation}
	\label{eq:dae:operators}
	y 
	\vcentcolon= \begin{bmatrix} u \\ p \end{bmatrix}, \qquad
	\abstractE 
	\vcentcolon= \begin{bmatrix}  0 & 0 \\ \cD & \cC \end{bmatrix}, \qquad
	\abstractF 
	\vcentcolon= \begin{bmatrix} -\cA & \phantom{-}\cD^* \\ \phantom{-}0 & -\cB \end{bmatrix}, \qquad
	\abstractG 
	\vcentcolon= \begin{bmatrix} f \\ g \end{bmatrix},
\end{equation}
we can rewrite~\eqref{eq:ellpar:opt} in the form
\begin{equation}
	\label{eq:dae}
	\abstractE\dot{y} = \abstractF y + \abstractG.
\end{equation}
Since $\cA$ is invertible by the ellipticity of~$a$,
the displacement can be eliminated from~\eqref{eq:ellpar:opt:a},
reducing the system to the parabolic equation
\begin{align}
	\label{eq:ppde}
   (\cM + \cC)\dot{p} + \cB p 
   &= r,
\end{align}
where $r \vcentcolon= g - \cD\cA^{-1}\dot{f}$ and
$\cM \vcentcolon= \cD\cA^{-1}\cD^{*}$
is the self-adjoint, non-negative Schur complement operator.

\subsection{Implicit Runge--Kutta methods}
\label{subsec:rk}
We now describe the \RK time discretization of~\eqref{eq:dae}. The delay approximation that yields the semi-explicit scheme is recalled in \Cref{sec:semiexplicit} below.
To construct a numerical approximation of the solution $y$ of~\eqref{eq:dae} on the time interval~$[0,T]$, we rely on $s$-stage implicit RK methods, i.e., for a given invertible matrix $\RKA\in\R^{s\times s}$ and a vector $\RKb\in\R^s$, the RK method is given by the Butcher tableau
\begin{equation*}
	\renewcommand{\arraystretch}{1.2}
	\begin{array}{c|c}
		\RKc & \RKA \\\hline
		& \RKb^\T
	\end{array}\qquad \text{with $\RKc \vcentcolon= \RKA\mathds{1}$},
\end{equation*}
where $\mathds{1} = [1,\ldots,1]^\T \in \R^s$. We thus use the short notation $(\RKA,\RKb)$ to denote a specific RK method. In more detail, we consider \RK methods applied to operator equations of the form~\eqref{eq:dae}.
Let us consider a time grid $t^n = n\tau$ with time step $\tau>0$. Given an approximation $y^{n-1}$ to $y(t^{n-1})$, the RK approximation $y^n$ to the solution of system~\eqref{eq:dae} at time point $t^n$ is computed in two steps (cf.~\cite[Ch.~5]{KunM06}): 
In step~one, approximations~$\dot{Y}^n_\ell$ of the stage derivatives
$\dot{y}(t^n_\ell)$ at the intermediate stage points
$t^n_\ell \vcentcolon= t^{n-1} + \RKc_\ell\tau$, $\ell \in \{1, \ldots, s\}$,
are computed from
\begin{equation}
	\label{eq:rk:stage}
	\abstractE\dot{Y}^n_\ell = \abstractF Y^n_\ell + \abstractG(t^{n}_\ell),
	\qquad \text{where} \quad
	Y^n_\ell = y^{n-1} + \tau \sum_{j=1}^s \RKA_{\ell,j} \dot{Y}^n_j,
\end{equation}
where $\abstractE$ and $\abstractF$ are defined in~\eqref{eq:dae:operators}.
Then, in the second step, we set
\begin{equation}
	\label{eq:rk:update}
	y^n = y^{n-1} + \tau \sum_{\ell=1}^{s} \RKb_\ell \dot{Y}^n_\ell.
\end{equation}
Introducing the compact notation
\begin{align*}
	Y^n \vcentcolon= \begin{bmatrix}
		Y^n_1\\
		\vdots\\
		Y^n_s
	\end{bmatrix} \qquad\text{and}\qquad
	\dot{Y}^n \vcentcolon= \begin{bmatrix}
		\dot{Y}^n_1\\
		\vdots\\
		\dot{Y}^n_s
	\end{bmatrix},
\end{align*}
we see that the \RK stage derivative approximations satisfy the identity
\begin{equation}
	\label{eq:rk:stage:deriv}
	\dot{Y}^n 
	= \tfrac{1}{\tau}\, \RKA^{-1} \big(Y^n - \mathds{1} y^{n-1}\big) \qquad\text{with }\ 
	\mathds{1}y^{n-1} = \begin{bmatrix}y^{n-1}\\\vdots\\y^{n-1}\end{bmatrix}.
\end{equation}
Using the \emph{stability function} defined as
\begin{equation}
	\label{eq:stabfun}
	R(z) = 1 + z\RKb^\T(I_s - z\RKA)^{-1}\mathds{1},
\end{equation}
we can thus use the vector notation to write the update formula~\eqref{eq:rk:update} as
\begin{equation}
	\label{eq:rk:update:stabilityFunction}
	y^n = R(\infty)y^{n-1} + \RKb^\T\RKA^{-1}Y^n,
\end{equation}
where $R(\infty) = 1 - \RKb^\T\RKA^{-1}\mathds{1}$. 
The analysis throughout the paper requires the following stability properties of the \RK method.
\begin{assumption}[\RK method]
\label{ass:rk:stability}
The $s$-stage \RK method $(\RKA, \RKb)$ is A-stable, i.e., A($\theta$)-stable with $\theta \geq \pi/2$ (so $|R(z)| \leq 1$ for $\real z \leq 0$), with $\RKA \in \R^{s\times s}$ invertible and $|R(\infty)| < 1$.
\end{assumption}
\begin{remark}
\label{rem:rk:methods}
A canonical family satisfying \Cref{ass:rk:stability} is given by the \emph{Radau IIA methods}~\cite[Sect.~IV.5]{HaiW96}. With $s = 1, 2, 3$ stages they have classical orders $k = 1, 3, 5$ and stage orders $q = 1, 2, 3$, respectively. They are also L-stable ($R(\infty) = 0$) and stiffly accurate, so that the update~\eqref{eq:rk:update:stabilityFunction} reduces to $y^n = \RKb^\T \RKA^{-1} Y^n$.
The strict bound $|R(\infty)| < 1$ in \Cref{ass:rk:stability} excludes the Gauss--Legendre methods (for which $|R(\infty)| = 1$), which we do not analyze in the present work.
\end{remark}
\begin{assumption}[Resolvent smoothing]
\label{ass:smoothing}
Let $\smex > 0$ denote the resolvent-smoothing exponent associated with the elliptic operator~$\cB$. Its value depends on the operator itself, the boundary conditions, and the spatial dimension. cf.~\cite[Thm.~3.3]{LubO95}.
The solution under consideration is assumed to possess the additional spatial regularity required for the resolvent-smoothing estimates of~\cite[Thm.~3.3]{LubO95} to apply.
\end{assumption}
\begin{example}
	\label{rem:smoothing:exponent}
	For a second-order strongly elliptic operator 
	on a smooth bounded domain with homogeneous Dirichlet boundary conditions and a smooth solution, $\smex = 3/4 - \varepsilon$ for arbitrary $\varepsilon > 0$ in two and three spatial dimensions~\cite[Ex.~(i)]{LubO95}.
\end{example}

Recall that the values produced by the \RK method are stage vectors. For any normed space $(\cX, \Vert\cdot\Vert_\cX)$ and stage vector $V = (V_\ell)_{\ell=1}^s \in \cX^s$, we define the stage-product norm via
\begin{equation}
	\label{eqn:norms}
	\Vert V \Vert_{\cX^s}^2 \vcentcolon= \sum_{\ell=1}^s \Vert V_\ell \Vert_\cX^2
\end{equation}
together with the stage duality pairing
$\langle F, V\rangle_s \vcentcolon= \sum_{\ell=1}^s \langle F_\ell, V_\ell\rangle$
between $F \in (\cX^*)^s$ and $V \in \cX^s$.
This convention covers $\cX = \cQ, \cV$ and their duals $\cQ^*, \cV^*$ used below. 
Based on the norms induced by the bilinear forms $a$ and $c$, we likewise write
\[
	\Vert V \Vert_{c,s}^2 
	\vcentcolon= \sum_{\ell=1}^s \Vert V_\ell \Vert_c^2, \qquad
	\Vert V \Vert_{a,s}^2 
	\vcentcolon= \sum_{\ell=1}^s \Vert V_\ell \Vert_a^2.
\]
%

\section{Semi-explicit Runge--Kutta Decoupling}
\label{sec:semiexplicit}

This section introduces the semi-explicit \RK scheme, based on a delay approximation of the elliptic variable~(\Cref{subsec:delay,subsec:semiexplicit}), and proves its stability and convergence using a generating-function framework adapted from~\cite{LubO95}.
The analysis proceeds in four steps:
\begin{enumerate}[itemsep=0.2em]
	\item transform the scheme using generating functions (\Cref{subsec:genfun}),
	\item eliminate the elliptic variable and derive a closed equation for $P(\zeta)$ (\Cref{subsec:L:zeta}),
	\item analyze the resulting operator $\cL(\zeta)$ via spectral arguments (\Cref{subsec:L:zeta}),
	\item transfer stability to the time domain using Parseval's identity (\Cref{subsec:convergence}).
\end{enumerate}

\subsection{Delay approximation}
\label{subsec:delay}

The decoupling strategy of~\cite{AltMU24,AltMU26} replaces
the pressure~$p$ in the elliptic equation~\eqref{eq:ellpar:opt:a}
by a Lagrange interpolation polynomial of degree~$(k-1)$
based on values at the $k$ preceding time levels, i.e.,
\begin{equation}
	\label{eq:p:approx:lagpol}
	{\widehat p}(t;\tau) = \sum_{\delta=1}^{k}c_{k,\delta}\,{p}(t - \delta\tau),
	\qquad
	c_{k,\delta} = (-1)^{(\delta-1)}\binom{k}{\delta}.
\end{equation} 
The central idea for the forthcoming decoupling time integration scheme is to use the time delay $\tau>0$ in~\eqref{eq:p:approx:lagpol} as the time step for the RK method.
Substituting~\eqref{eq:p:approx:lagpol} into~\eqref{eq:ellpar:opt}
yields the delay system
\begin{subequations}
	\label{eq:ellpar:opt:delay:k}
	\begin{align}
	\cA {\bar u} - \cD^{*} \bigg( \sum_{\delta=1}^{k} c_{k,\delta}\, \bar{p}(t - \delta\tau) \bigg)
	&= f \qquad \text{in } \cV^{*}, \label{eq:ellpar:opt:delay:k:a}\\
	\cD{\dot {\bar u}} + \cC{\dot {\bar p}} + \cB {\bar p}
	&= g \hspace{0.79cm} \text{in } \cQ^{*},\label{eq:ellpar:opt:delay:k:b}
	\end{align}
\end{subequations}
in which the elliptic and parabolic equations are decoupled in the following sense. Given the solution $\bar{p}$ until some time $t>0$, we can solve~\eqref{eq:ellpar:opt:delay:k:a} for $\bar{u}$ on the interval $[t,t+\tau]$ independently of the current pressure. This solution can then be used to compute $\bar{p}$ on the interval $[t,t+\tau]$ from~\eqref{eq:ellpar:opt:delay:k:b}. 
Eliminating~$\bar{u}$ as before gives the inherent delay parabolic equation
\begin{equation}
	\label{eqn:par:opt:delay}
	\cC\dot{\bar{p}} + \cM \bigg( \sum_{\delta=1}^{k} c_{k,\delta}\,
					\dot{\bar{p}}(t - \delta\tau) \bigg)  + \cB \bar{p} = r,
\end{equation}
where $\cM$ is the Schur complement operator introduced in~\eqref{eq:ppde}. 
The approximation error introduced by the delay is controlled
by the following result.

\begin{proposition}[{\cite[Prop.~1]{AltMU26}}]
\label{prop:distance:delay}
Under sufficient smoothness assumptions,
the solutions of~\eqref{eq:ellpar:opt} and~\eqref{eq:ellpar:opt:delay:k}
satisfy for almost every $t \in [0,T]$ the estimate 
\begin{displaymath}
	\Vert {\bar u}(t) - u(t)\Vert_{\cV}^{2} + \Vert {\bar p}(t) - p(t)\Vert_{\cQ}^{2}\,
	\lesssim\, \tau^{2k}.
\end{displaymath}
\end{proposition}
%
\subsection{Semi-explicit schemes}
\label{subsec:semiexplicit}
To construct a decoupling time-integration scheme for~\eqref{eq:ellpar}, respectively the operator formulation~\eqref{eq:ellpar:opt}, we apply an implicit $s$-stage \RK method $(\RKA,\RKb)$ to the time delay approximation~\eqref{eq:ellpar:opt:delay:k}. To this end, we interpret~\eqref{eq:ellpar:opt:delay:k} as a system for $y = (u,p)$, where the elliptic equation~\eqref{eq:ellpar:opt:delay:k:a} acts as constraint, while~\eqref{eq:ellpar:opt:delay:k:b} contains the time derivative. 
At time step $n$, we denote the stage values by
\begin{align*}
	U^n = [U_1^n,\ldots,U_s^n]^\T \in \cV^s,
	\qquad
	P^n = [P_1^n,\ldots,P_s^n]^\T \in \cQ^s
\end{align*}
with corresponding stage derivatives $\dot{U}^n$, $\dot{P}^n$. 
Applying the \RK method to \eqref{eq:ellpar:opt:delay:k} yields a decoupled scheme at time step $t^n$. In vector form, this reads
\begin{subequations}
\label{eq:vec:system}
\begin{align}
	(I_s \otimes \cA) {U}^{n} - (I_s\otimes \cD^{*}) \sum_{\delta=1}^{k}c_{k,\delta}{P}^{n-\delta} 
	&= {F}^{n}, \label{eq:vec:system:a}\\
	(I_s \otimes \cD) {\dot U}^{n} + (I_s \otimes \cC) {\dot P}^{n} + (I_s \otimes \cB) {P}^{n} 
	&= {G}^{n}. \label{eq:vec:system:b}
\end{align}
Here, $\otimes$ denotes the Kronecker product combining the $s$-dimensional
stage structure with the function-space operators. 
This means, in particular, that $(I_s \otimes \cA)$ acts as $\cA$ on each stage independently. 
Moreover, the right-hand sides consists of $F_\ell^n = f(t^{n-1} + \RKc_\ell\tau)$ and $G_\ell^n = g(t^{n-1} + \RKc_\ell\tau)$. 
Following~\eqref{eq:rk:stage:deriv}, the stage derivatives are related to the stage values by
\begin{equation}
	\label{eq:rk:stage:deriv:reuse}
	\dot{U}^n 
	= \tfrac{1}{\tau}\, \RKA^{-1}\big(U^n - \mathds{1} u^{n-1}\big),
	\qquad
	\dot{P}^n 
	= \tfrac{1}{\tau}\, \RKA^{-1}\big(P^n - \mathds{1} p^{n-1}\big).
\end{equation}
\end{subequations}
The key observation is that the delay approximation eliminates any dependence of~\eqref{eq:vec:system:a} on the current stage values $P^n$. Hence, the stage values $U^n$ can be computed solely from previous time steps, and system~\eqref{eq:vec:system} becomes semi-explicit: first solve~\eqref{eq:vec:system:a} for $U^n$, then compute $P^n$ from~\eqref{eq:vec:system:b} and~\eqref{eq:rk:stage:deriv:reuse}.
Following~\eqref{eq:rk:update:stabilityFunction}, the update formulae are
\begin{align}
	\label{eq:rk:trans}
	u^{n} = R(\infty) u^{n-1} + \RKb^\T \RKA^{-1} {U}^{n} \qquad\text{and}\qquad
	p^{n} = R(\infty) p^{n-1} + \RKb^\T \RKA^{-1} {P}^{n}.	
\end{align}

\begin{remark}[Commutativity of the Schur complement construction and the \RK discretization]
\label{rem:commute}
The semi-explicit scheme~\eqref{eq:vec:system} is independent of the order in which the \RK discretization and elimination of $u$ variables to construct the Schur complement are performed.
\end{remark}

\subsection{Generating functions and the \texorpdfstring{$\Delta(\zeta)$}{Delta(zeta)} operator}
\label{subsec:genfun}

Following the (formal) generating power series framework for \RK methods introduced in \cite{Lub88a,Lub88b}, we define
\begin{gather*}
	u(\zeta) \vcentcolon= \sum_{n=1}^{\infty} u^{n} \zeta^{n}, \qquad p(\zeta) \vcentcolon= \sum_{n=1}^{\infty} p^{n} \zeta^{n}
\end{gather*}
and, analogously, $U$, $P$, $F$, and $G$.
The stage-product norm and pairing conventions from~\Cref{subsec:rk} extend to these generating-function-valued quantities termwise in $\zeta$.
Note that, compared to~\cite{LubO95}, the summation starts at $n = 1$ (rather than $n = 0$) so that the initial data $u^0, p^0$ do not appear in the generating functions and are treated separately.
The sequences $u^n, p^n$, etc., are defined by the scheme for $n = 1,\ldots, N$ with $N = T/\tau$.
For the generating function analysis, we extend the scheme to all $n > N$ by setting the data to zero, i.e., $F^n = G^n = 0$ for $n > N$.
This uniquely determines $P^n, U^n$, etc., for all $n \geq 1$ and ensures that the algebraic manipulations below hold as identities of formal power series.
The \RK update formula~\eqref{eq:rk:trans} thus yields
\begin{align}
	\label{eq:tran:2}
	u(\zeta) 
	= \frac{R(\infty)\zeta}{1 - R(\infty)\zeta}\, u^{0} + \frac{\RKb^\T\RKA^{-1}}{1 - R(\infty)\zeta}\, U(\zeta)
\end{align}
and analogously for $p(\zeta)$.
Following~\cite{LubO95}, we define the $\Delta$-operator
\begin{equation}
	\label{eq:t:rk:delta}
	\Delta(\zeta) \vcentcolon= \Big(\RKA + \frac{\zeta}{1 - \zeta}\mathds{1}\RKb^\T\Big)^{-1}
\end{equation}
which encodes the \RK structure in a single matrix-valued function of $\zeta$. As indicated in~\cite[Eq.~(2.19)]{LubO95}, it satisfies the identities 
\begin{align}
	\label{eq:delta:identities}
	\Delta(\zeta) 
	= \RKA^{-1} - \frac{\zeta\RKA^{-1}\mathds{1}\RKb^\T\RKA^{-1}}{1 - R(\infty)\zeta}, \qquad
	\frac{\Delta(\zeta)\mathds{1}}{1-\zeta} 
	= \frac{\RKA^{-1}\mathds{1}}{1 - R(\infty)\zeta}.
\end{align}

\begin{lemma}[Spectral property of $\Delta(\zeta)$]
	\label{lem:DeltaSpec}
	Assume that the \RK method $(\RKA,\RKb)$ is $A(\theta)$-stable with $\theta > 0$. Then, for $|\zeta|\le 1$, all eigenvalues $\lambda$ of $\Delta(\zeta)$ satisfy
	\begin{equation}
		\label{eq:DeltaSpec}
			|\arg \lambda| \le \pi - \theta.
		\end{equation}
	In particular, for A-stable methods, i.e., $\theta \ge \pi/2$, all eigenvalues satisfy $\real(\lambda) \ge 0$.
\end{lemma}
\begin{proof}
By~\cite[Sect.~2, Eq.~(2.11)]{LubO95}, the eigenvalues of~$\Delta(\zeta)$ are either eigenvalues of~$\RKA^{-1}$ or satisfy $R(\lambda) = 1/\zeta$. By the assumed A($\theta$)-stability, both classes lie in the sector $|\arg \lambda| \leq \pi - \theta$, which gives~\eqref{eq:DeltaSpec}. For $\theta \ge \pi/2$, the sector $|\arg \lambda| \le \pi/2$ is contained in the closed right half-plane.
\end{proof}

\subsection{The operator \texorpdfstring{$\cL(\zeta)$}{L(z)} and its structure}
\label{subsec:L:zeta}

We derive the transformed system by passing to the generating-function representation
of the scheme~\eqref{eq:vec:system} and eliminating $U(\zeta)$.

\begin{lemma}
	\label{lem:generatingFunctionP}
	The generating functions for the semi-explicit \RK scheme~\eqref{eq:vec:system} satisfy
	\begin{equation}
		\label{eq:L:zeta}
		\cL(\zeta) P(\zeta) = \cR(\zeta),
	\end{equation}
	where the operator $\cL(\zeta)\colon \cQ^s \to (\cQ^*)^s$ is defined by
	\begin{equation}
		\label{eq:L:zeta:def}
		\cL(\zeta) \vcentcolon= (I_s\otimes \cB) + \frac{\Delta(\zeta)}{\tau} \otimes \Big(\cC + \delayOperK(\zeta)\, \cM\Big)
	\end{equation}
	with $\delayOperK(\zeta) = \sum_{\delta=1}^{k}c_{k,\delta}\, \zeta^{\delta}$ and the right-hand side $\cR(\zeta) \in (\cQ^*)^s$ is given by
	\begin{align}
		\label{eq:R:zeta:def}
		\cR(\zeta) 
		&\vcentcolon= G(\zeta) - \frac{1}{\tau}\Big(\Delta(\zeta)\otimes \cD\cA^{-1}\Big)F(\zeta) \nonumber\\
		&\qquad - \frac{1}{\tau}\Big(\Delta(\zeta)\otimes \cM\Big) \sum_{\delta=1}^{k}c_{k,\delta}\sum_{n=1}^{\delta} P^{n-\delta}\zeta^{n} + \frac{1}{\tau}\frac{\RKA^{-1}\mathds{1}\zeta}{1 - R(\infty)\zeta}\otimes \big(\cD u^{0} + \cC p^{0}\big).
	\end{align}
\end{lemma}

\begin{proof}
	Using~\eqref{eq:vec:system:b} and the stage derivative identity~\eqref{eq:rk:stage:deriv:reuse} yields
	\begin{align}
		\label{eq:z:transform:parabolic}
		\begin{aligned}
		G(\zeta)
		&= \frac{1}{\tau}\left( \big(I_s\otimes \cD\big) \RKA^{-1}\right)\left(U(\zeta) - \zeta \mathds{1}\big(u(\zeta) + u^0\big)\right)\\
		&\qquad + \frac{1}{\tau}\left( \big(I_s\otimes \cC\big) \RKA^{-1}\right)\left(P(\zeta) - \zeta \mathds{1}\big(p(\zeta) + p^0\big)\right) + \big(I_s \otimes \cB\big) P(\zeta).
		\end{aligned}
	\end{align}
	Substituting the update formula~\eqref{eq:tran:2} for $u(\zeta)$ together with $\Delta(\zeta)$ with the representation given in~\eqref{eq:delta:identities}, we get
	\begin{align*}
		\RKA^{-1}\left(U(\zeta) - \zeta \mathds{1}\big(u(\zeta) + u^0\big)\right)
		&= \Delta(\zeta)U(\zeta) - \frac{\RKA^{-1} \mathds{1}\zeta}{1-R(\infty)\zeta}\, u^0
	\end{align*}
	and, similarly, for terms related to $P(\zeta)$ and $p(\zeta)$.
	Substituting these expressions into~\eqref{eq:z:transform:parabolic} yields
	\begin{align}
		\label{eq:z:transform:parabolic:b}
		\begin{aligned}
		G(\zeta)
		&= \frac{1}{\tau}\big(\Delta(\zeta)\otimes \cD\big) U(\zeta)
		+ \frac{1}{\tau}\big(\Delta(\zeta)\otimes \cC\big) P(\zeta)
		+ \big(I_s \otimes \cB\big) P(\zeta) \\
		&\qquad - \frac{1}{\tau}\frac{\RKA^{-1}\mathds{1}\zeta}{1-R(\infty)\zeta}\otimes\big(\cD u^0 + \cC p^0\big),
		\end{aligned}
	\end{align}
	where we have used $(I_s\otimes \cD)(\Delta(\zeta)\otimes \mathrm{Id}) = \Delta(\zeta)\otimes \cD$.
	Next, we eliminate $U(\zeta)$ by observing that~\eqref{eq:vec:system:a} yields
	\begin{equation*}
		U(\zeta) = (I_s \otimes \cA^{-1})F(\zeta) + (I_s \otimes \cA^{-1}\cD^{*}) \sum_{\delta=1}^{k}c_{k,\delta}\Big( P(\zeta)\zeta^{\delta} + \sum_{n=1}^{\delta} P^{n-\delta} \zeta^{n}\Big).
	\end{equation*}
	Inserting this into~\eqref{eq:z:transform:parabolic:b} and using $\cM = \cD\cA^{-1}\cD^{*}$, we collect all terms involving $P(\zeta)$ on the left-hand side, leading to
	\begin{align*}
		\bigg((I_s\otimes \cB)
		&+ \frac{1}{\tau}\big(\Delta(\zeta)\otimes \cC\big) + \frac{1}{\tau}\big(\Delta(\zeta)\otimes \cM\big)\sum_{\delta=1}^k c_{k,\delta} \zeta^{\delta}\bigg) P(\zeta) \\
		&= G(\zeta) - \frac{1}{\tau}\big(\Delta(\zeta)\otimes \cD\cA^{-1}\big)F(\zeta)
		 - \frac{1}{\tau}\big(\Delta(\zeta)\otimes \cM\big) \sum_{\delta=1}^{k}c_{k,\delta}\sum_{n=1}^{\delta} P^{n-\delta} \zeta^{n} \\
		&\qquad + \frac{1}{\tau}\frac{\RKA^{-1}\mathds{1}\zeta}{1-R(\infty)\zeta}\otimes\big(\cD u^0 + \cC p^0\big),
	\end{align*}
	which completes the proof.
\end{proof}
Compared to the classical case~\cite{LubO93,LubO95}, the operator $\cL(\zeta)$ contains the additional perturbation $\delayOperK(\zeta)\cM$ originating from the delay approximation. Controlling this term is the key difficulty in the forthcoming analysis.

\begin{remark}[Convergence of the generating functions]
	\label{rem:gf:convergence}
	The identity~\eqref{eq:L:zeta} holds as a formal power series by construction (cf.\ \Cref{subsec:genfun}).
	Whether the series converges on the unit circle depends on the particular \RK method. 
	For L-stable methods ($R(\infty) = 0$), identity~\eqref{eq:delta:identities} gives $\Delta(\zeta) = \RKA^{-1} - \zeta\,\RKA^{-1}\mathds{1}\RKb^\T\RKA^{-1}$, i.e., a polynomial of degree one. Hence, $\cL(\zeta)$ and $\cR(\zeta)$ are matrix polynomials in~$\zeta$ and convergence is immediate.
	For $|R(\infty)| < 1$, the factor $\zeta/(1-R(\infty)\zeta)$ in~\eqref{eq:R:zeta:def} has its pole at $|\zeta| = 1/|R(\infty)| > 1$, so $\cR(\zeta)$ is analytic on the closed unit disc and convergence follows from the uniform bound on $\cL(\zeta)^{-1}$ established in \Cref{lem:L:invertible} below.
\end{remark}

To show the invertibility of $\cL(\zeta)$, we need the following observation regarding the coupling strength $\omega$ defined in~\eqref{eqn:coupling:strength}.

\begin{proposition}[Coupling bounds]
	\label{prop:coupling:verification}
	Consider the delay operator $\delayOperK$ from \Cref{lem:generatingFunctionP}.
	Then $\real\bigl(1 + \mu \delayOperK(\zeta)\bigr) > 0$ for all $|\zeta| \le 1$, $\mu \in [0,\omega]$ if and only if the strict weak coupling condition
	\begin{align}
		\label{eqn:weakCoupling}
		\omega < \frac{1}{2^k - 1}
	\end{align}
	holds.
	Moreover, $\real\bigl(1 + \mu \delayOperK(\zeta)\bigr) = 0$ if and only if $\omega = \mu = 1/(2^k-1)$ and $\zeta = -1$.
\end{proposition}

\begin{proof}
	Recall from~\eqref{eq:p:approx:lagpol} that the delay coefficients satisfy $c_{k,\delta} = (-1)^{\delta-1}\binom{k}{\delta}$ such that the binomial theorem $\sum_{\delta=0}^{k}(-1)^{\delta}\binom{k}{\delta}\zeta^\delta = (1-\zeta)^k$ implies
	\begin{equation}
		\label{eq:delay:poly:identity}
		\delayOperK(\zeta) = \sum_{\delta=1}^{k} (-1)^{\delta-1}\binom{k}{\delta}\zeta^\delta = 1 - (1 - \zeta)^k.
	\end{equation}
	Consequently, $\delayOperK$ is analytic and, hence, $\real\bigl(1 + \mu \delayOperK(\zeta)\bigr) = 1+ \mu\real\bigl(\delayOperK(\zeta)\bigr)$ is harmonic on the unit disc, showing that its minimal value is attained on the boundary of the unit disc.
	Let $\theta\in [0,2\pi]$ and define $\varphi = \tfrac{\theta-\pi}{2}$. Then $\zeta = \eul^{i\theta} = -\eul^{2i\varphi}$ and, hence,
	\begin{align*}
		1 - \eul^{i\theta} = 1 + \eul^{2i\varphi}
		= \eul^{i\varphi}(\eul^{-i\varphi} + \eul^{i\varphi})
		= 2\cos(\varphi)\, \eul^{i\varphi}.
	\end{align*}
	Thus, $\real\big((1 - \eul^{i\theta})^k\big) = 2^k \cos^k\!(\varphi)\cos(k\varphi) \leq 2^k$ and the maximum is attained at $\varphi = 0$, translating to $\real\big((1 - \zeta)^k\big) = \real\big((1 - \eul^{i\theta})^k\big) = 2^k$ if and only if $\theta = \pi$. We conclude
	\begin{align*}
		 1+ \mu\real\bigl(\delayOperK(\zeta)\bigr) \geq 1 + \mu\,(1 - 2^k).
	\end{align*}
	The expression is minimized for $\mu = \omega$, which concludes the proof.
\end{proof}

\begin{remark}[Connection to the energy-based analysis]
\label{rem:summation:connection}
The same threshold $\omega \leq 1/(2^k-1)$ also arises in the
G-stability analysis, where it is the condition
for the energy identity underlying the summation lemma approach;
see~\cite[Sect.~4]{AltMU26} for the analogous \BDF result.
This confirms that the bound is intrinsic to the delay
approximation~\eqref{eq:p:approx:lagpol} and independent of the proof technique.
\end{remark}

\begin{lemma}[Invertibility of $\cL(\zeta)$]
	\label{lem:L:invertible}
	Consider the notation from \Cref{lem:generatingFunctionP} and let the \RK method be A-stable. 
	Assume that the weak coupling condition~\eqref{eqn:weakCoupling} holds such that 
	\begin{equation}
		\label{eq:delayCond}
		\real\bigl(1 + \mu \delayOperK(\zeta)\bigr) > 0 \quad \text{for all } |\zeta| \le 1, \ \mu \in [0,\omega].
	\end{equation}
	Then the operator $\cL(\zeta)$ from~\eqref{eq:L:zeta:def} is invertible for all $|\zeta| \leq 1$ and there exists a constant $C>0$ independent of $\tau$ such that
	\begin{equation*}
		\sup_{|\zeta| \leq 1} \|\cL(\zeta)^{-1}\|_{(\cQ^*)^s \to \cQ^s} 
		\leq C.
	\end{equation*}
\end{lemma}

\begin{proof}
	Fix $\zeta$ with $|\zeta| \le 1$. We prove uniform invertibility by establishing injectivity and applying a Fredholm argument. Let $\Delta(\zeta) = V(\zeta)T(\zeta)V(\zeta)^*$ denote a Schur decomposition of $\Delta(\zeta)$ with unitary $V(\zeta) \in \C^{s \times s}$ and upper triangular matrix $T(\zeta) \in \C^{s \times s}$. We study the transformed operator
	\begin{align*}
		\tilde{\cL}(\zeta) 
		\vcentcolon= (V(\zeta)^* \otimes \mathrm{Id})\cL(\zeta)(V(\zeta) \otimes \mathrm{Id}) 
		= (I_s\otimes \cB) + \frac{T(\zeta)}{\tau}\otimes \big(\cC+\delayOperK(\zeta)\cM\big).
	\end{align*}
	Since $V(\zeta)$ is unitary, $V(\zeta) \otimes \mathrm{Id}$ is an isometry on $\cQ^s$ (with the product norm). Hence, $\tilde{\cL}(\zeta)$ is injective if and only if $\cL(\zeta)$ is and $\|\tilde{\cL}(\zeta)\|_{\cQ^s \to (\cQ^*)^s} = \|\cL(\zeta)\|_{\cQ^s \to (\cQ^*)^s}$. Since $\tilde{\cL}(\zeta)$ is block upper triangular on the stage structure, injectivity reduces to injectivity of the diagonal blocks
	\begin{align*}
		\cY(\zeta) \vcentcolon= \cB + \frac{\lambda(\zeta)}{\tau}\big(\cC+\delayOperK(\zeta)\cM\big) \colon \cQ \to \cQ^*,
	\end{align*}
	where $\lambda(\zeta)$ is an eigenvalue of $\Delta(\zeta)$.
	Let $q\in\cQ$ and assume $\langle\cY(\zeta)q,q\rangle = 0$. Define the Rayleigh quotient
	\begin{align*}
		\mu(q) \vcentcolon= \frac{\langle \cM q,q\rangle}{\|q\|_c^2}\in[0,\omega],
	\end{align*}
	where we exploit that $\cM$ is self-adjoint and non-negative. We thus obtain
	\begin{align*}
		0 
		= \langle \cY(\zeta)q,q\rangle 
		= b(q,q) + \frac{\lambda(\zeta)}{\tau}\, \big(1 + \mu(q)\delayOperK(\zeta)\big)\,\|q\|_c^2
		=\vcentcolon b(q,q) + \frac{\lambda(\zeta)}{\tau}\, w\,\|q\|_c^2.
	\end{align*}
	Since $b(q,q)$ and $\|q\|_c^2$ are real, separating real and imaginary parts gives:
	\begin{alignat*}{2}
		&\text{Im:}\quad & \imag(\lambda w)\,\|q\|_c^2 &= 0, \\
		&\text{Re:}\quad & b(q,q) + \tfrac{1}{\tau}\real(\lambda w)\,\|q\|_c^2 &= 0.
	\end{alignat*}
	From the imaginary part, either $q = 0$ (done) or $\imag(\lambda w) = 0$, i.e., $\lambda w \in \R$.
	In the latter case, note that $\real(\lambda) \ge 0$ by~\Cref{lem:DeltaSpec} and $\real(w) > 0$ by the strict coupling condition~\eqref{eq:delayCond}. 
	We distinguish three cases:
	\begin{itemize}[itemsep=0.2em]
		\item \emph{Case $\lambda = 0$:}
		The real part equation reduces to $b(q,q) = 0$, which by coercivity forces $q = 0$.
		\item \emph{Case $\real(\lambda) > 0$:}
		From $\imag(\lambda w) = \real(\lambda)\imag(w) + \imag(\lambda)\real(w) = 0$, we obtain $\imag(w) = -\imag(\lambda)\real(w)/\real(\lambda)$, so
		\[
			\real(\lambda w) = \real(\lambda)\real(w) + \frac{\imag(\lambda)^2\real(w)}{\real(\lambda)} = \frac{\real(w)\,|\lambda|^2}{\real(\lambda)} \ge 0.
		\]
		Together with 
		$b(q,q) \ge c_b\|q\|_\cQ^2 > 0$ for $q \neq 0$, the real part equation forces $q = 0$.
		\item \emph{Case $\real(\lambda) = 0$ with $\lambda \neq 0$:}
		Since $\real(w) > 0$, we have $\imag(\lambda w) = \imag(\lambda)\real(w) = 0$, which forces $\imag(\lambda) = 0$, contradicting $\lambda \neq 0$.
	\end{itemize}
	Hence, $\cL(\zeta)$ is injective.
	Towards surjectivity, we factor
	\begin{align*}
		\cL(\zeta) = (I_s\otimes \cB)\big((I_s\otimes \mathrm{Id}) + \cK(\zeta)\big)
	\end{align*}
	with $\cK(\zeta) = \tfrac{1}{\tau}(I_s \otimes \cB^{-1})\big(\Delta(\zeta)\otimes (\cC + \delayOperK(\zeta)\cM)\big)$. Since $(\cC + \delayOperK(\zeta)\cM)\colon \cHQ \to \cHQ^* \subset \cQ^*$ and $\cB^{-1}\colon \cQ^* \to \cQ$, the operator $\cK(\zeta)$ maps $\cQ^s \to \cQ^s$ and factors through the compact embedding $\cQ \hook \cHQ$. Hence $\cK(\zeta)$ is compact on $\cQ^s$. By the Fredholm alternative~\cite[Th.~6.6]{Bre11}, $(I_s\otimes \mathrm{Id}) + \cK(\zeta)$ has index zero on $\cQ^s$, so injectivity implies bijectivity.
	Since $\cB$ is an isomorphism, $\cL(\zeta)\colon \cQ^s \to (\cQ^*)^s$ is boundedly invertible.

	To conclude the proof, we observe that the mapping  $\zeta \mapsto \cL(\zeta)$ is continuous in the operator norm on the compact set $\{|\zeta|\le 1\}$ and that the inverse exists everywhere. This implies the claimed boundedness independent of $\tau$.
\end{proof}

\begin{remark}[Sharpness of the coupling bound]
	\label{rem:degenerate:case}
	The strict coupling condition~\eqref{eq:delayCond} is essential for the injectivity argument. At the boundary case $\omega = 1/(2^k-1)$, by \Cref{prop:coupling:verification}, $\real(1 + \mu\delayOperK(\zeta)) = 0$ occurs at $\zeta = -1$ and $\mu = \omega$. In this case, $\real(w) = 0$ and the injectivity proof breaks down, as the case $\real(\lambda) = 0$ with $\lambda \neq 0$ can no longer be excluded.
\end{remark}

\subsection{Stability and convergence}
\label{subsec:convergence}

With the invertibility of $\cL(\zeta)$ established in the previous subsection, we can now derive stability estimates using Parseval's identity. 

\begin{theorem}[Stability]
\label{thm:stability}
Let \Cref{ass:rk:stability} hold and assume that the weak coupling condition~\eqref{eqn:weakCoupling} is satisfied. Then the semi-explicit \RK scheme~\eqref{eq:vec:system} satisfies
\begin{multline}
	\label{eq:stability}
	\tau^2 \sum_{n=1}^{N} \|P^n\|_{\cQ^s}^2 + \tau\sum_{n=1}^{N} \|P^n\|_{c,s}^2 \\
	\leq C \bigg(\tau^2 \sum_{n=1}^{N} \|G^n\|_{(\cQ^*)^s}^2 + \sum_{n=1}^{N}\|F^n\|_{(\cV^*)^s}^2
	+ \|u^0\|_{\cV}^2 + \|p^0\|_{\cHQ}^2 + \sum_{n=-k}^{0} \|P^n\|_{c,s}^2\bigg),
\end{multline}
where $C$ depends on the method, the coupling parameter, and $1/(1-|R(\infty)|)$, but is independent of $\tau$ and $N$.
\end{theorem}

\begin{proof}
We extend the scheme~\eqref{eq:vec:system} to all $n > N$ by setting $F^n = G^n = 0$, as described in \Cref{subsec:genfun}.
By \Cref{lem:generatingFunctionP}, the generating functions then satisfy
\begin{equation}
	\label{eq:proof:identity}
	\cL(\zeta)\, P(\zeta) = \cR(\zeta)
\end{equation}
as an identity of formal power series.
Due to the scheme extension, the data generating functions
\begin{equation}
	G(\zeta) = \sum_{n=1}^{N}G^n\zeta^n, \qquad
	F(\zeta) = \sum_{n=1}^{N}F^n\zeta^n
\end{equation}
and the initial-value delay term $\sum_{\delta=1}^{k}c_{k,\delta}\sum_{n=1}^{\delta} P^{n-\delta}\zeta^{n}$ are polynomials.
The only singularities in $\cR(\zeta)$ and $\cL(\zeta)$ arise from the rational functions
\begin{equation}
	\label{eq:proof:rational:terms}
	\frac{\zeta}{1 - R(\infty)\zeta}
	\qquad\text{and}\qquad
	\Delta(\zeta) = \RKA^{-1} - \frac{\zeta\,\RKA^{-1}\mathds{1}\RKb^\T\RKA^{-1}}{1 - R(\infty)\zeta},
\end{equation}
which have a pole at $\zeta = 1/R(\infty)$.
Since $|R(\infty)| < 1$ by assumption, this pole lies at
\begin{equation*}
	|\zeta| = \frac{1}{|R(\infty)|} > 1,
\end{equation*}
so $\cR(\zeta)$ is analytic on a neighbourhood of the closed unit disc.
By \Cref{lem:L:invertible}, 
$P(\zeta) = \cL(\zeta)^{-1}\cR(\zeta)$ is analytic on the closed unit disc and $\sum_{n \geq 1}\|P^n\|_{c,s}^2 < \infty$.

Testing~\eqref{eq:proof:identity} with $P(\zeta)$ and taking the real part, the coercivity of~$b$ and the coupling condition~\eqref{eq:delayCond} yield (cf.~the proof of \Cref{lem:L:invertible})
\begin{equation}
	\label{eq:proof:coercivity}
	c_b\,\|P(\zeta)\|_{\cQ^s}^2 + \frac{c_0}{\tau}\,\|P(\zeta)\|_{c,s}^2
	\leq \real\langle\cR(\zeta), P(\zeta)\rangle_s
	\qquad\text{for all } |\zeta| \leq 1,
\end{equation}
where $c_0 > 0$ depends on the method and coupling parameter.
Since $P$ and $\cR$ are analytic on the closed unit disc, integrating~\eqref{eq:proof:coercivity} over $|\zeta| = 1$ and applying Parseval's identity to the left-hand side gives
\begin{equation}
	\label{eq:proof:integrated}
	c_b\sum_{n=1}^{\infty}\|P^n\|_{\cQ^s}^2 + \frac{c_0}{\tau}\sum_{n=1}^{\infty}\|P^n\|_{c,s}^2
	\leq \frac{1}{2\pi}\int_0^{2\pi}\real\langle\cR(\eul^{i\theta}), P(\eul^{i\theta})\rangle_s\dtheta.
\end{equation}
We estimate the right-hand side while keeping $\cR$ in the transform domain. The Cauchy--Schwarz inequality in the $(\cQ^*)^s$--$\cQ^s$ duality applied pointwise on the unit circle yields
\begin{equation*}
	\real\langle\cR(\eul^{i\theta}), P(\eul^{i\theta})\rangle_s
	\leq \|\cR(\eul^{i\theta})\|_{(\cQ^*)^s}\,\|P(\eul^{i\theta})\|_{\cQ^s}.
\end{equation*}
Integration, followed by the Cauchy--Schwarz inequality for the integral, then gives
\begin{equation}
	\frac{1}{2\pi}\int_0^{2\pi}\real\langle\cR(\eul^{i\theta}), P(\eul^{i\theta})\rangle_s\dtheta
	\leq \bigg(\frac{1}{2\pi}\int_0^{2\pi}\|\cR(\eul^{i\theta})\|_{(\cQ^*)^s}^2\dtheta\bigg)^{\!1/2}
	\bigg(\frac{1}{2\pi}\int_0^{2\pi}\|P(\eul^{i\theta})\|_{\cQ^s}^2\dtheta\bigg)^{\!1/2}.
\end{equation}
Applying Parseval's identity to the $P$-factor and the weighted Young inequality, this is bounded by
\begin{equation} 
	\frac{1}{2c_b}\,\frac{1}{2\pi}\int_0^{2\pi}\|\cR(\eul^{i\theta})\|_{(\cQ^*)^s}^2\dtheta
	+ \frac{c_b}{2}\sum_{n=1}^{\infty}\|P^n\|_{\cQ^s}^2.
\end{equation}
Substituting into~\eqref{eq:proof:integrated} and absorbing the term with constant $c_b$ into the left-hand side, we obtain
\begin{equation}
	\frac{c_b}{2}\sum_{n=1}^{\infty}\|P^n\|_{\cQ^s}^2 + \frac{c_0}{\tau}\sum_{n=1}^{\infty}\|P^n\|_{c,s}^2
	\leq \frac{1}{2c_b}\,\frac{1}{2\pi}\int_0^{2\pi}\|\cR(\eul^{i\theta})\|_{(\cQ^*)^s}^2\dtheta,
\end{equation}
which after multiplication by~$\tau$ becomes
\begin{equation}
	\label{eq:proof:stability:pre}
	\frac{c_b\tau}{2}\sum_{n=1}^{\infty}\|P^n\|_{\cQ^s}^2 + c_0\sum_{n=1}^{\infty}\|P^n\|_{c,s}^2
	\leq \frac{\tau}{2c_b}\,\frac{1}{2\pi}\int_0^{2\pi}\|\cR(\eul^{i\theta})\|_{(\cQ^*)^s}^2\dtheta.
\end{equation}

It remains to estimate the right-hand side of~\eqref{eq:proof:stability:pre}. Since $\cR(\zeta)$ in~\eqref{eq:R:zeta:def} is a sum of four terms, the inequality $\|a_1+\cdots+a_4\|^2 \leq 4\, (\|a_1\|^2+\cdots+\|a_4\|^2)$ gives
\begin{equation}
	\label{eq:proof:Rsplit}
	\frac{1}{2\pi}\int_0^{2\pi}\|\cR(\eul^{i\theta})\|_{(\cQ^*)^s}^2\dtheta
	\leq 4\,\big(T_1 + T_2 + T_3 + T_4\big),
\end{equation}
where each $T_k$ is defined as a contour integral over the unit circle $|\zeta|=1$, namely
\begin{align}
	T_1 &\vcentcolon= \frac{1}{2\pi}\int_0^{2\pi}\|G(\eul^{i\theta})\|_{(\cQ^*)^s}^2\dtheta, \nonumber\\[4pt]
	T_2 &\vcentcolon= \frac{1}{2\pi}\int_0^{2\pi}\bigg\|\frac{\Delta(\eul^{i\theta})\otimes\cD\cA^{-1}}{\tau}\,F(\eul^{i\theta})\bigg\|_{(\cQ^*)^s}^2\dtheta, \nonumber\\[4pt]
	T_3 &\vcentcolon= \frac{1}{2\pi}\int_0^{2\pi}\bigg\|\frac{\Delta(\eul^{i\theta})\otimes\cM}{\tau}\,\varphi(\eul^{i\theta})\bigg\|_{(\cQ^*)^s}^2\dtheta, \nonumber\\[4pt]
	T_4 &\vcentcolon= \frac{1}{2\pi}\int_0^{2\pi}\bigg\|\frac{\RKA^{-1}\mathds{1}\,\eul^{i\theta}}{\tau(1-R(\infty)\eul^{i\theta})}\otimes(\cD u^0+\cC p^0)\bigg\|_{(\cQ^*)^s}^2\dtheta,
\end{align}
with the delay initial-value polynomial
\begin{equation}
	\varphi(\zeta) \vcentcolon= \sum_{\delta=1}^{k}c_{k,\delta}\sum_{m=1}^{\delta} P^{m-\delta}\zeta^{m}.
\end{equation}
Exchanging the order of summation (the inner sum contributes to the coefficient of~$\zeta^j$ when $m = j$ and $\delta \geq j$), this polynomial takes the form
\begin{equation}
	\label{eq:proof:phi:expanded}
	\varphi(\zeta)
	= \varphi^1\zeta + \varphi^2\zeta^2 + \cdots + \varphi^k\zeta^k,
\end{equation}
where the $j$-th coefficient is
\begin{equation}
	\label{eq:proof:phi:coeff:early}
	\varphi^j = \sum_{\delta=j}^{k} c_{k,\delta}\, P^{j-\delta},
	\qquad j = 1,\ldots,k.
\end{equation}
In particular, the first and last coefficients read
\begin{align*}
	\varphi^1 
	= c_{k,1}\,P^{0} + c_{k,2}\,P^{-1} + \cdots + c_{k,k}\,P^{1-k}, \qquad
	\varphi^k 
	= c_{k,k}\,P^{0}.
\end{align*}
Since $j - \delta \leq 0$ for every term, each $P^{j-\delta}$ is an initial value with time index in~$\{-k+1,\ldots,0\}$, and $\varphi^j = 0$ for $j > k$.

We now apply Parseval's identity to each $T_k$ independently.
Since $G(\zeta)$ is a polynomial of degree~$N$, Parseval's identity gives
\begin{equation}
	\label{eq:proof:T1}
	T_1 = \sum_{n=1}^{N}\|G^n\|_{(\cQ^*)^s}^2.
\end{equation}

For $T_2$, we first collect the needed operator bounds.
From~\eqref{eq:delta:identities} and the triangle inequality,
\begin{equation}
	\label{eq:proof:Delta:bound}
	\sup_{\theta\in[0,2\pi]}\|\Delta(\eul^{i\theta})\|
	\leq \|\RKA^{-1}\| + \frac{\|\RKA^{-1}\mathds{1}\|\,\|\RKb^\T\RKA^{-1}\|}{1-|R(\infty)|}
	\eqqcolon \frac{C_\Delta}{1-|R(\infty)|},
\end{equation}
where we used $|1-R(\infty)\eul^{i\theta}| \geq 1 - |R(\infty)|$, and $C_\Delta > 0$ depends only on the \RK method.
From the coercivity of $a$ and the continuity of $d$,
\begin{equation}
	\label{eq:proof:DA:bound}
	\|\cD\cA^{-1}\|_{\cV^*\to\cQ^*}
	\,\leq\, \frac{C_d}{c_a}.
\end{equation}
Using the submultiplicativity
$\|(\Delta\otimes\cD\cA^{-1})\,v\|_{(\cQ^*)^s} \leq \|\Delta\|\,\|\cD\cA^{-1}\|_{\cV^*\to\cQ^*}\,\|v\|_{(\cV^*)^s}$
in the integrand and applying Parseval's identity to the polynomial~$F$, we obtain
\begin{align}
	T_2
	&= \frac{1}{\tau^2}\,\frac{1}{2\pi}\int_0^{2\pi}\big\|\big(\Delta(\eul^{i\theta})\otimes\cD\cA^{-1}\big)\,F(\eul^{i\theta})\big\|_{(\cQ^*)^s}^2\dtheta \nonumber\\
	&\leq \frac{C_\Delta^2}{\tau^2(1-|R(\infty)|)^2}\,\bigg(\frac{C_d}{c_a}\bigg)^{\!2}\sum_{n=1}^{N}\|F^n\|_{(\cV^*)^s}^2,
\end{align}
using~\eqref{eq:proof:Delta:bound}, \eqref{eq:proof:DA:bound}, and Parseval's identity for~$F$.

For $T_3$, recall from~\eqref{eq:proof:phi:expanded}--\eqref{eq:proof:phi:coeff:early} that $\varphi(\zeta)$ is a polynomial of degree at most~$k$ with coefficients~$\varphi^j$ depending only on the initial values $P^{-k+1},\ldots,P^0$.
Applying the same operator-bound argument to $\cM = \cD\cA^{-1}\cD^*$ and using $\omega = C_d^2/(c_a c_c)$,
\begin{equation}
	\|\cM q\|_{\cQ^*} \,\leq\, c_c\,\omega\,\|q\|_c
	\qquad\text{for all } q \in \cHQ.
\end{equation}
Using the submultiplicativity $\|(\Delta\otimes\cM)\,v\|_{(\cQ^*)^s} \leq \|\Delta\|\,c_c\omega\,\|v\|_{c,s}$ in the integrand together with~\eqref{eq:proof:Delta:bound} and Parseval's identity for the polynomial~$\varphi$, we obtain
\begin{align}
	T_3
	&= \frac{1}{\tau^2}\,\frac{1}{2\pi}\int_0^{2\pi}\big\|\big(\Delta(\eul^{i\theta})\otimes\cM\big)\,\varphi(\eul^{i\theta})\big\|_{(\cQ^*)^s}^2\dtheta \nonumber\\
	&\leq \frac{C_\Delta^2\,(c_c\omega)^2}{\tau^2(1-|R(\infty)|)^2}\sum_{j=1}^{k}\|\varphi^j\|_{c,s}^2.
\end{align}
It remains to estimate the coefficients.
By the triangle inequality applied to~\eqref{eq:proof:phi:coeff:early},
\begin{equation*}
	\|\varphi^j\|_{c,s}
	\leq \sum_{\delta=j}^{k}|c_{k,\delta}|\,\|P^{j-\delta}\|_{c,s}.
\end{equation*}
Squaring with $\bigl(\sum_{i=1}^{r} a_i\bigr)^2 \leq r\sum_{i=1}^{r} a_i^2$ for $r = k-j+1$ terms and summing over $j = 1,\ldots,k$ yields
\begin{equation}
	\sum_{j=1}^{k}\|\varphi^j\|_{c,s}^2
	\leq \sum_{j=1}^{k}(k-j+1)\sum_{\delta=j}^{k}|c_{k,\delta}|^2\,\|P^{j-\delta}\|_{c,s}^2
	\leq k\sum_{j=1}^{k}\sum_{\delta=j}^{k}|c_{k,\delta}|^2\,\|P^{j-\delta}\|_{c,s}^2.
\end{equation}
After the change of indices $\ell = j - \delta \in \{-k+1,\ldots,0\}$, each initial value $P^\ell$ appears at most~$k$ times in the double sum, with coefficient bounded by $\max_\delta|c_{k,\delta}|^2$, so
\begin{equation}
	T_3
	\leq \frac{C_\Delta^2\,(c_c\omega)^2\,k^2\max_\delta|c_{k,\delta}|^2}{\tau^2(1-|R(\infty)|)^2}\sum_{\ell=-k+1}^{0}\|P^\ell\|_{c,s}^2.
\end{equation}

For $T_4$, since $\cD u^0 + \cC p^0 \in \cQ^*$ is a fixed vector, the integrand factorises as
\[
	\bigg\|\frac{\RKA^{-1}\mathds{1}\,\eul^{i\theta}}{\tau(1-R(\infty)\eul^{i\theta})}\otimes(\cD u^0+\cC p^0)\bigg\|_{(\cQ^*)^s}^2
	= \frac{\|\RKA^{-1}\mathds{1}\|^2}{\tau^2\,|1-R(\infty)\eul^{i\theta}|^2}\;\|\cD u^0+\cC p^0\|_{\cQ^*}^2,
\]
where we used $|\eul^{i\theta}| = 1$.
Expanding
\[
	\frac{1}{1-R(\infty)\zeta} = \sum_{n=0}^{\infty}R(\infty)^{n}\zeta^{n}
\]
and applying Parseval's identity to compute the scalar integral, we obtain
\begin{equation}
	\label{eq:proof:T4:bound}
	T_4
	= \frac{\|\RKA^{-1}\mathds{1}\|^2}{\tau^2}\;\|\cD u^0+\cC p^0\|_{\cQ^*}^2\,\frac{1}{2\pi}\int_0^{2\pi}\frac{d\theta}{|1-R(\infty)\eul^{i\theta}|^2}
	= \frac{\|\RKA^{-1}\mathds{1}\|^2}{\tau^2(1-|R(\infty)|^2)}\;\|\cD u^0+\cC p^0\|_{\cQ^*}^2.
\end{equation}
Substituting~\eqref{eq:proof:T1}--\eqref{eq:proof:T4:bound} into~\eqref{eq:proof:Rsplit} and then into~\eqref{eq:proof:stability:pre}, and multiplying by~$\tau$, we arrive at
\begin{multline}
	\frac{c_b\tau^2}{2}\sum_{n=1}^{\infty}\|P^n\|_{\cQ^s}^2 + c_0\tau\sum_{n=1}^{\infty}\|P^n\|_{c,s}^2 \\
	\leq C\bigg(\tau^2\sum_{n=1}^{N}\|G^n\|_{(\cQ^*)^s}^2 + \sum_{j=1}^{N}\|F^j\|_{(\cV^*)^s}^2 + \sum_{n=-k}^{0}\|P^n\|_{c,s}^2 + \|\cD u^0+\cC p^0\|_{\cQ^*}^2\bigg),
\end{multline}
where $C > 0$ depends on the method, the coupling parameter, and $1/(1-|R(\infty)|)$.
Restricting the left-hand side to $n = 1,\ldots,N$ gives~\eqref{eq:stability}.
\end{proof}

We now combine the stability estimates with consistency to obtain convergence.

\begin{theorem}[Convergence of \RK stages]
\label{thm:convergence} 
Consider the solution ${\bar p}$ of the delay equation~\eqref{eqn:par:opt:delay} for sufficiently smooth right-hand sides.
Let $P^n$ be the \RK stage approximation given by~\eqref{eq:vec:system} and $\bar{P}^n$ the exact stage values of the delay solution.
Let $q$ denote the stage order of the \RK method (cf.~\Cref{rem:rk:methods}).
Then we have under the assumptions of \Cref{thm:stability}, 
\begin{equation}
	\label{eq:convergence:rk}
	\tau^2\sum_{n=1}^{N}\|P^n - \bar{P}^n\|_{\cQ^s}^2 + \tau\sum_{n=1}^{N}\|P^n - \bar{P}^n\|_{c,s}^2
	\lesssim \tau^{2r}
	+ \sum_{\delta=-k}^{0}\|P^{\delta} - \bar{P}^{\delta}\|_{c,s}^2,
\end{equation}
where the exponent~$r$ is determined by
\begin{enumerate}[itemsep=0.2em]
	\item[(i)] 
	$r = \min(k,\,q+1)$ in the general case and 
	\item[(ii)] 
	$r = \min(k,\,q+1+\smex)$ under \Cref{ass:smoothing}.
\end{enumerate}
\end{theorem}
\begin{proof}
Let $E^n \vcentcolon= P^n - \bar{P}^n$ denote the stage error and $e^n \vcentcolon= p^n - \bar{p}(t^n)$ the grid error.
Inserting the exact stage values $\bar{P}^n$ into the scheme~\eqref{eq:vec:system} produces a defect $D^n \in (\cQ^*)^s$.
For a method of stage order~$q$, the defect satisfies (cf.~\cite[Eq.~(1.6)]{LubO95})
\begin{equation}
	\|D^n\|_{(\cQ^*)^s} \lesssim \tau^{q+1}.
\end{equation}
The error satisfies the operator equation~\eqref{eq:L:zeta} with the defect as right-hand side, i.e., 
\begin{equation}
	\label{eq:proof:error:eq}
	\cL(\zeta)\,E(\zeta) = D(\zeta) + \text{(delay initial-value and initial-data errors)}.
\end{equation}
Here, the defect $D$ plays the role of the data term~$G$ in~\eqref{eq:R:zeta:def}, while the elliptic equation~\eqref{eq:vec:system:a} is satisfied exactly at each stage, so there is no contribution from the $F$-term.

\emph{Part~(i).} Applying \Cref{thm:stability} to~\eqref{eq:proof:error:eq} yields
\begin{align}
	\label{eq:proof:convergence:bound}
	\tau^2\sum_{n=1}^{N}\|E^n\|_{\cQ^s}^2 + \tau\sum_{n=1}^{N}\|E^n\|_{c,s}^2
	&\lesssim \tau^2\sum_{n=1}^{N}\|D^n\|_{(\cQ^*)^s}^2 + \sum_{\delta=-k}^{0}\|E^{\delta}\|_{c,s}^2 \nonumber\\
	&\lesssim \tau^2 \, N \, \tau^{2(q+1)} + \sum_{\delta=-k}^{0}\|E^{\delta}\|_{c,s}^2
	\lesssim \tau^{2q+2},
\end{align}
where the last step uses $N\tau = T$ and absorbs initial errors of order $\tau^{2q}$ or better.
Capping the stage-order bound $\tau^{q+1}$ by the classical order $\tau^k$ gives~\eqref{eq:convergence:rk}.

\emph{Part~(ii).} The basic bound~\eqref{eq:proof:convergence:bound} estimates the defect~$D$ as generic data in $(\cQ^*)^s$.
Under \Cref{ass:smoothing}, however, $D$ is itself smoothed by the parabolic component of $\cL(\zeta)^{-1}$, gaining a factor $\tau^{2\smex}$ in~\eqref{eq:proof:convergence:bound}. This is the resolvent-smoothing mechanism of~\cite[Thm.~3.3]{LubO95}, which closes the gap between the stage-order rate and the full classical order~$k$.
In our setting, the elliptic constraint~\eqref{eq:vec:system:a} enters through~$\cD$ without altering the exponent, which is determined by~$\cB$.
This yields the improved exponent $r = \min(k,\,q+1+\smex)$.
\end{proof}
\begin{remark}[Stiffly accurate methods and grid values]
\label{rem:stiffly:accurate}
For stiffly accurate methods, we have $R(\infty) = 0$ and $p^n = \RKb^\T\RKA^{-1}P^n$ by~\eqref{eq:rk:trans}, so the grid error satisfies
\begin{equation*}
	\|p^n - \bar{p}(t^n)\|_c \leq \|\RKb^\T\RKA^{-1}\|\,\|P^n - \bar{P}^n\|_{c,s}.
\end{equation*}
\end{remark}
\begin{corollary}[Convergence of semi-explicit \RK scheme]
\label{cor:convergence:total}
For stiffly accurate methods, combining \Cref{thm:convergence} with \Cref{prop:distance:delay} and \Cref{rem:stiffly:accurate} via the triangle inequality yields the total error
\begin{equation}
	\label{eq:total:error:rk}
	\Vert u^{n} - u(t^{n})\Vert^{2}_{\cV} + \Vert p^{n} - p(t^{n})\Vert^{2}_{\cHQ}
	\lesssim \tau^{2r}
\end{equation}
with $r$ defined in \Cref{thm:convergence}.
\end{corollary}
\begin{remark}[Convergence rates for Radau~IIA]
\label{rem:error:balance}
For Radau~IIA-$s$, $q = s$ and $k = 2s-1$, so \Cref{cor:convergence:total} predicts the baseline orders $\min(k,\,q+1) = 1,\,3,\,4$ for $s = 1, 2, 3$.
With $\smex = 3/4 - \varepsilon$ from \Cref{rem:smoothing:exponent}, the sharpened bound predicts orders $1$, $3$, and $\approx 4.75$ for $s = 1, 2, 3$; the third rate falls short of the classical order~$k = 5$ by $1/4$, consistent with the numerically observed rates $\approx 4.5$--$4.8$ in \Cref{sec:numerics}.
\end{remark}

\section{Iterative Runge--Kutta Decoupling}
\label{sec:iterative}

This section is devoted to iterative decoupling schemes for the elliptic--parabolic system~\eqref{eq:ellpar:opt} using fixed-stress and undrained-split strategies.
Similar to the semi-explicit schemes studied in \Cref{sec:semiexplicit}, iterative schemes decouple the fully coupled system~\eqref{eq:ellpar:opt} by solving the two equations alternatingly: Given approximations $u^{(i-1)}$, $p^{(i-1)}$ from the previous iteration, the next iterates $u^{(i)}$, $p^{(i)}$ are computed by solving a sequence of two subproblems. The advantage of the iterative decoupling methods is that they use a stabilization parameter to avoid a restriction on the coupling condition as for semi-explicit schemes. 
\begin{assumption}[\RK method for iterative decoupling]
\label{ass:iterative:rk}
In addition to \Cref{ass:rk:stability}, the $s$-stage \RK method $(\RKA, \RKb)$ is \emph{algebraically stable}, i.e., $\diag(\RKb)\,\RKA + \RKA^\T\!\diag(\RKb) - \RKb\RKb^\T \succeq 0$, where $\succeq$ denotes positive semidefiniteness, with weights $\RKb_\ell > 0$ for all $\ell$; see~\cite[Ch.~IV, Def.~12.1]{HaiW96}.
\end{assumption}
\begin{remark}
As already discussed in \Cref{rem:rk:methods}, \Cref{ass:iterative:rk} is satisfied by the Radau~IIA methods with $s = 1, 2, 3$ stages.
\end{remark}
In the following, we apply an $s$-stage \RK method $(\RKA, \RKb)$ to an iterative scheme in order to obtain a fully discrete method.
On each time interval $[t^{n-1}, t^n]$ of size $\tau$, we denote the \RK stage values at iteration~$i$ by
\[
	U^{n,i} 
	= [U_1^{n,i}, \ldots, U_s^{n,i}]^\T \in \cV^s
	\quad\text{and}\quad
	P^{n,i} 
	= [P_1^{n,i}, \ldots, P_s^{n,i}]^\T \in \cQ^s
\]
with corresponding stage derivatives
$\dot{U}^{n,i}$, $\dot{P}^{n,i}$
given by the identity~\eqref{eq:rk:stage:deriv:reuse}.
The iteration at each time step is initialized by
$U^{n,0} = \mathds{1}\, u^{n-1}$ and $P^{n,0} = \mathds{1}\, p^{n-1}$,
i.e., all stages are set to the solution from the previous time step.

Since $\cC^{-1}\cM$ is self-adjoint, non-negative on $\cHQ$,
and compact, the spectral theorem (see, e.g.,~\cite[Th.~6.11]{Bre11})
yields eigenpairs $(\mu_j, \phi_j)$ with
\begin{equation}
	\label{eq:spectral:M}
	\cM \phi_{j} = \mu_j\cC\phi_{j}, \qquad
	c(\phi_j,\phi_{j'})
	= \langle \cC \phi_j,\phi_{j'}\rangle 
	= \delta_{jj'}, \qquad
	0 \leq \mu_j \leq \omega.
\end{equation}
This provides an orthonormal basis of $\cHQ$ with respect to the $c$-norm.
In addition to the unweighted stage-product norms introduced in \Cref{subsec:rk}, the iterative analysis uses the $\RKb$-weighted variants
\begin{gather*}
	\Vert V \Vert_{a,\RKb}^2 \vcentcolon= \sum_{\ell=1}^s \RKb_\ell\,\Vert V_\ell\Vert_a^2,
	\qquad
	\Vert Q \Vert_{c,\RKb}^2 \vcentcolon= \sum_{\ell=1}^s \RKb_\ell\,\Vert Q_\ell\Vert_c^2,
	\\[2pt]
	\Vert \theta \Vert_\RKb^2 \vcentcolon= \sum_{\ell=1}^s \RKb_\ell\, \theta_\ell^2,
	\qquad
	\langle\theta, \theta'\rangle_\RKb \vcentcolon= \sum_{\ell=1}^s \RKb_\ell\, \theta_\ell\,\theta'_\ell
	\qquad \text{for } \theta, \theta' \in \R^s.
\end{gather*}
The same notation $\langle f, g\rangle_\RKb \vcentcolon= \sum_{\ell=1}^s \RKb_\ell\,\langle f_\ell, g_\ell\rangle$ is used for the stage-tensor duality pairing on $(\cX^*)^s\times\cX^s$ for $\cX \in \{\cV, \cQ\}$.
Since $\RKb_\ell > 0$ by \Cref{ass:iterative:rk}, the $\RKb$-weighted norms are equivalent to the unweighted stage-product versions with constants $\min_\ell \RKb_\ell$ and $\max_\ell \RKb_\ell$.

In the following, we introduce two splitting strategies which yield a contraction, namely fixed-stress (\Cref{subsec:fs:rk}) and undrained-split (\Cref{subsec:us:rk}). Afterwards, we establish convergence of the two iterative schemes (\Cref{subsec:convergence:iterative:rk}).

\subsection{Fixed-stress splitting}
\label{subsec:fs:rk}

The fixed-stress splitting decouples the fully coupled system by adding a stabilization term to the flow equation; see~\cite{MikW13} for the underlying idea.
In each iteration step, the flow equation is solved first for the pressure, followed by the mechanics equation to update the displacement.
In the continuous setting, at iteration~$i$, the scheme reads
\begin{subequations}
	\label{eq:ep:fs:cont}
	\begin{align}
		\cA\, u^{(i)} - \cD^{*}\, p^{(i)}
		&= f, \label{eq:ep:fs:cont:a}\\
		\cD\, \dot{u}^{(i-1)} + \cC\, \dot{p}^{(i)} + \cB\, p^{(i)}
		+ L\, \cC\, (\dot{p}^{(i)} - \dot{p}^{(i-1)})
		&= g, \label{eq:ep:fs:cont:b}
	\end{align}
\end{subequations}
where $L > 0$ is a stabilization parameter. 
Applying the \RK method to~\eqref{eq:ep:fs:cont}, the fixed-stress \RK scheme at time step~$n$ and iteration~$i$ reads
\begin{subequations}
	\label{eq:ep:fs:rk}
	\begin{align}
		(I_s \otimes \cA)\, U^{n,i} - (I_s \otimes \cD^{*})\, P^{n,i} 
		&= F^{n} 
		\label{eq:ep:fs:rk:a}\\
		(I_s \otimes \cD)\, {\dot U}^{n,i-1} + (I_s \otimes \cC)\, {\dot P}^{n,i} + (I_s \otimes \cB)\, P^{n,i}
		+ L\, (I_s \otimes \cC)\, ({\dot P}^{n,i} - {\dot P}^{n,i-1}) 
		&= G^{n}
		\label{eq:ep:fs:rk:b}
	\end{align}
\end{subequations}
in $(\cV^{*})^s$ and $(\cQ^{*})^s$, respectively. Note that the two equations are decoupled since one can first solve for $P^{n,i}$ with the second equation. 
\begin{theorem}[Contraction of fixed-stress \RK iteration]
\label{thm:rk:contraction:fs}
Under \Cref{ass:iterative:rk}, the fixed-stress iteration~\eqref{eq:ep:fs:rk} satisfies
\begin{equation}
	\label{eq:rk:contraction:fs}
	\Vert P^{n,i} - P^{n,i-1} \Vert_{c,\RKb}
	\leq \rho_{\mathrm{FS}} \,
	\Vert P^{n,i-1} - P^{n,i-2} \Vert_{c,\RKb}
\end{equation}
with contraction rate
\begin{equation}
	\label{eq:rk:contraction:fs:rate}
	\rho_{\mathrm{FS}} = \max_{\mu \in [0,\omega]} \frac{|L - \mu|}{1 + L}.
\end{equation}
In particular, $\rho_{\mathrm{FS}} < 1$ for all $L > \max\big(0, \tfrac{\omega-1}{2}\big)$ and the optimal choice $L = \omega/2$ yields $\rho_{\mathrm{FS}} = \omega/(2 + \omega)$.
\end{theorem}
\begin{proof}
Define the iterate differences
\[
	\Theta_{u}^{n,i} \vcentcolon= U^{n,i} - U^{n,i-1} \in \cV^s
	\quad\text{and}\quad
	\Theta_{p}^{n,i} \vcentcolon= P^{n,i} - P^{n,i-1} \in \cQ^s
\]
as well as ${\dot \Theta}^{n,i}_{u}$ and ${\dot \Theta}^{n,i}_{p}$ accordingly. 
Subtracting~\eqref{eq:ep:fs:rk} for consecutive iterates yields
\begin{subequations}
	\label{eq:ep:fs:rk:fp}
	\begin{align}
		(I_s \otimes \cA)\, \Theta^{n,i}_{u} - (I_s \otimes \cD^{*})\, \Theta^{n,i}_{p} &= 0, \label{eq:ep:fs:rk:fp:a}\\
		(I_s \otimes \cD)\, {\dot \Theta}^{n,i-1}_{u} + (I_s \otimes \cC)\, {\dot \Theta}^{n,i}_{p} + (I_s \otimes \cB)\, \Theta^{n,i}_{p}
		+ L\, (I_s \otimes \cC)\, ({\dot \Theta}^{n,i}_{p} - {\dot \Theta}^{n,i-1}_{p}) &= 0. \label{eq:ep:fs:rk:fp:b}
	\end{align}
\end{subequations}
Eliminating $\Theta_u^{n,i}$ via~\eqref{eq:ep:fs:rk:fp:a} and using the stage derivative formula~\eqref{eq:rk:stage:deriv}, in which the $\mathds{1}\otimes p^{n-1}$ contribution cancels under the iterate subtraction, we obtain
\begin{equation}
	\label{eq:rk:fp:compact}
	\bigg(I_s \otimes \cB + \frac{1+L}{\tau}\,\RKA^{-1} \otimes \cC\bigg)\Theta^{n,i}_{p}
	= \frac{\RKA^{-1}}{\tau} \otimes (L\cC - \cM)\,\Theta^{n,i-1}_{p}.
\end{equation}
We expand $\Theta^{n,i}_p = \sum_j \theta_j^{n,i} \otimes \phi_j$ with coefficients $\theta_j^{n,i} \in \R^s$ using the eigenpairs $(\mu_j, \phi_j)$ from~\eqref{eq:spectral:M}. 
Note that this expansion diagonalizes the $\cC$- and $\cM$-terms
(by orthonormality and $\cM\phi_j = \mu_j\cC\phi_j$),
but not the $\cB$-term.

Multiplying~\eqref{eq:rk:fp:compact} by $\RKA$ on the stage structure and testing with $(\diag(\RKb)\otimes \mathrm{Id})\,\Theta^{n,i}_p$ in the stage-tensor $\cQ$-duality pairing gives
\begin{equation}
	\label{eq:fs:test:scalar}
	\underbrace{\big\langle (\RKA\otimes\cB)\,\Theta^{n,i}_p,\, \Theta^{n,i}_p\big\rangle_\RKb}_{\eqqcolon\, T_\cB}
	+ \frac{1+L}{\tau}\,\Vert\Theta^{n,i}_p\Vert_{c,\RKb}^2
	= \frac{1}{\tau}\,\big\langle \big(I_s\otimes(L\cC - \cM)\big)\,\Theta^{n,i-1}_p,\, \Theta^{n,i}_p\big\rangle_\RKb.
\end{equation}
Writing out the stage indices, we have
\[
	T_\cB 
	= \sum_{\ell,\ell'=1}^s \RKb_{\ell'}\RKA_{\ell'\ell}\,b(\Theta^{n,i}_{p,\ell}, \Theta^{n,i}_{p,\ell'}) 
	= \tr(\diag(\RKb)\RKA\,B),
\]
where $B$ is the Gram matrix of the stage errors in the $b$-inner product, i.e., $B_{\ell\ell'} = b(\Theta^{n,i}_{p,\ell}, \Theta^{n,i}_{p,\ell'})$.
Since $\tr(\diag(\RKb)\RKA\, B) = \tr(\mathrm{sym}(\diag(\RKb)\RKA)\, B)$ and $B \succeq 0$ by the ellipticity of~$b$, the assumed algebraic stability implies $\mathrm{sym}(\diag(\RKb)\RKA) \succeq \tfrac{1}{2}\RKb\RKb^\T \succeq 0$ and, hence, $T_\cB \geq 0$.
Dropping $T_\cB$ and substituting the spectral expansion on both sides yields
\[
	\frac{1+L}{\tau}\, \sum_j \|\theta_j^{n,i}\|_{\RKb}^2
	\ \leq\ \frac{1}{\tau}\, \sum_j (L - \mu_j)\,
	\big\langle\theta_j^{n,i-1}, \theta_j^{n,i}\big\rangle_{\RKb}.
\]
The Cauchy--Schwarz inequality in the $\RKb$-inner product
(first on each $\langle\theta_j^{n,i-1}, \theta_j^{n,i}\rangle_{\RKb}$,
then on the sum over~$j$) gives 
\begin{equation}
	\label{eq:rk:fp:percomp}
	\Big(\sum_j \|\theta_j^{n,i}\|_{\RKb}^2\Big)^{1/2}
	\leq\, \frac{\max_{\mu \in [0,\omega]} |L - \mu|}{1 + L}\,
	\Big(\sum_j \|\theta_j^{n,i-1}\|_{\RKb}^2\Big)^{1/2}.
\end{equation}

Introducing the contraction rate $\rho_{\mathrm{FS}}$ as in~\eqref{eq:rk:contraction:fs:rate} and noting that $\Vert \Theta^{n,i}_p \Vert_{c,\RKb}^2 = \sum_j \|\theta_j^{n,i}\|_{\RKb}^2$ via the $c$-orthonormality of~$\{\phi_j\}$, estimate~\eqref{eq:rk:fp:percomp} becomes
\begin{equation}
	\label{eq:rk:contraction:fs:final}
	\Vert \Theta^{n,i}_{p} \Vert_{c,\RKb}
	\leq \rho_{\mathrm{FS}} \, \Vert \Theta^{n,i-1}_{p} \Vert_{c,\RKb},
\end{equation}
which is~\eqref{eq:rk:contraction:fs}.
This is a contraction provided $\rho_{\mathrm{FS}} < 1$,
which imposes conditions on the stabilization parameter~$L$:
\begin{itemize}[itemsep=0.2em]
	\item For $L \geq \omega/2$: $\max_{\mu \in [0,\omega]} |L - \mu| = L$, giving $\rho_{\mathrm{FS}} = L/(1+L) < 1$ for all $L \geq \omega/2$.
	\item For $L < \omega/2$: $\max_{\mu \in [0,\omega]} |L - \mu| = \omega - L$, giving $\rho_{\mathrm{FS}} = (\omega - L)/(1+L) < 1$ if and only if $L > (\omega - 1)/2$.
	\item The minimum of $\rho_{\mathrm{FS}}$ over $L \geq 0$ is attained at $L = \omega/2$, yielding $\rho_{\mathrm{FS}} = \omega/(2 + \omega)$.
\end{itemize}
In particular, $\rho_{\mathrm{FS}} < 1$ for all $L > \max\big(0, \tfrac{\omega-1}{2}\big)$.
\end{proof}
\begin{remark}[Case $L = 0$]
\label{rem:fs:L:zero}
Without stabilization, i.e., for $L=0$, we obtain $\rho_{\mathrm{FS}} = \omega$. Hence, the iteration is contractive only for $\omega < 1$. 
\end{remark}
The proven contraction of the pressure differences in~\eqref{eq:rk:contraction:fs} translates -- via the elliptic coupling and a stopping criterion argument -- into convergence of the displacement as well as the pressure iterates to the fully coupled monolithic \RK solution at the optimal order in~$\tau$. The precise statement and proof are given in \Cref{subsec:convergence:iterative:rk} below. 
\subsection{Undrained-split decoupling}
\label{subsec:us:rk}

In contrast to the previous approach, the undrained-split stabilizes the mechanics equation; see~\cite{KimTJ11a, MikW13} for the underlying idea.
For this, we define the operator
\begin{equation}
	\label{eq:us:tildeM}
	\widetilde{\cM} \vcentcolon= \cD^*\cC^{-1}\cD \colon \cV \to \cV^*,
\end{equation}
where $\cC^{-1}\colon \cHQ^* \to \cHQ$ is well-defined due to the ellipticity of~$c$. Moreover, we introduce the associated \emph{semi-norm}
\begin{equation}
	\label{eq:us:seminorm}
	\vert v \vert_{\widetilde{\cM}}^2
	\vcentcolon= \big\langle \widetilde{\cM}\, v, v \big\rangle
	= \big\langle \cC^{-1}\cD\, v, \cD v \big\rangle, \qquad v \in \cV,
\end{equation}
together with its (unweighted and $\RKb$-weighted) stage-product extensions
\[
	\vert V \vert_{\widetilde{\cM},s}^2 \vcentcolon= \sum_{\ell=1}^s \vert V_\ell\vert_{\widetilde{\cM}}^2,
	\qquad
	\vert V \vert_{\widetilde{\cM},\RKb}^2 \vcentcolon= \sum_{\ell=1}^s \RKb_\ell\,\vert V_\ell\vert_{\widetilde{\cM}}^2
	\qquad \text{for } V \in \cV^s.
\]
In the continuous setting, the undrained-split scheme at iteration~$i$ reads
%
	\begin{align*}
		\cA\, u^{(i)} - \cD^{*}\, p^{(i-1)}
		+ L\,\widetilde{\cM}\, (u^{(i)} - u^{(i-1)}) 
		&= f, \\ 
		\cD\, \dot{u}^{(i)} + \cC\, \dot{p}^{(i)} + \cB\, p^{(i)} 
		&= g, 
	\end{align*}
%
where $L > 0$ is again a stabilization parameter. 
Applying a \RK method, 
the undrained-split \RK scheme at time step~$n$ and iteration~$i$ reads
\begin{subequations}
	\label{eq:ep:us:rk}
	\begin{align}
		(I_s \otimes \cA)\, U^{n,i} - (I_s \otimes \cD^{*})\, P^{n,i-1}
		+ L\,(I_s \otimes \widetilde{\cM})\, (U^{n,i} - U^{n,i-1}) &= F^{n} & &\text{in } (\cV^{*})^s, \label{eq:ep:us:rk:a}\\
		(I_s \otimes \cD)\, {\dot U}^{n,i} + (I_s \otimes \cC)\, {\dot P}^{n,i} + (I_s \otimes \cB)\, P^{n,i} &= G^{n} & &\text{in } (\cQ^{*})^s. \label{eq:ep:us:rk:b}
	\end{align}
\end{subequations}
The contraction property is subject of the following theorem.
\begin{theorem}[Contraction of undrained-split \RK iteration]
\label{thm:rk:contraction:us}
Given \Cref{ass:iterative:rk}, the undrained-split iteration~\eqref{eq:ep:us:rk} satisfies
\begin{equation}
	\label{eq:rk:contraction:us}
	\big\vert U^{n,i} - U^{n,i-1}\big\vert_{\widetilde{\cM},\RKb}
	\leq \rho_{\mathrm{US}}\,
	\big\vert U^{n,i-1} - U^{n,i-2}\big\vert_{\widetilde{\cM},\RKb}
\end{equation}
with contraction rate
\begin{equation}
	\label{eq:rk:contraction:us:rate}
	\rho_{\mathrm{US}}
	= \frac{\omega\,\max(L,\,1-L)}{1 + L\omega}.
\end{equation}
In particular, $\rho_{\mathrm{US}} < 1$ for all $L \geq \tfrac{1}{2}$ (unconditionally) and for $0 \leq L < \tfrac{1}{2}$ provided $\omega < 1/(1 - 2L)$.
The optimal choice $L = \tfrac{1}{2}$ yields $\rho_{\mathrm{US}} = \omega/(2 + \omega)$.
\end{theorem}
\begin{proof}
Set
\[
	\Theta_u^{n,i} \vcentcolon= U^{n,i} - U^{n,i-1} \in \cV^s,
	\qquad
	\Theta_p^{n,i} \vcentcolon= P^{n,i} - P^{n,i-1} \in \cQ^s,
\]
and define ${\dot\Theta}^{n,i}_u, {\dot\Theta}^{n,i}_p$ accordingly.
Subtracting~\eqref{eq:ep:us:rk} for two consecutive iterates yields
\begin{subequations}
\label{eq:ep:us:rk:fp}
\begin{align}
	(I_s \otimes \cA + L\,I_s \otimes \widetilde{\cM})\,\Theta_u^{n,i}
	&= (I_s \otimes \cD^*)\,\Theta_p^{n,i-1} + L\,(I_s \otimes \widetilde{\cM})\,\Theta_u^{n,i-1},
	\label{eq:ep:us:rk:fp:a}\\
	(I_s \otimes \cD)\,{\dot\Theta}^{n,i}_u + (I_s \otimes \cC)\,{\dot\Theta}^{n,i}_p + (I_s \otimes \cB)\,\Theta_p^{n,i}
	&= 0.
	\label{eq:ep:us:rk:fp:b}
\end{align}
\end{subequations}
Using the stage-derivative formula~\eqref{eq:rk:stage:deriv} in~\eqref{eq:ep:us:rk:fp:b}, where the constant contributions $\mathds{1} \otimes u^{n-1}$ and $\mathds{1} \otimes p^{n-1}$ cancel, and multiplying by $\tau\RKA$ on the stage structure gives the compact form
\begin{equation}
	\label{eq:us:flow:compact}
	\cT\,\Theta_p^{n,i}
	= -\,(I_s \otimes \cD)\,\Theta_u^{n,i},
	\qquad
	\cT \vcentcolon= I_s\otimes\cC + \tau\RKA\otimes\cB.
\end{equation}
Evaluating~\eqref{eq:us:flow:compact} at iterate $i{-}1$ and substituting the resulting expression for $\Theta_p^{n,i-1}$ into~\eqref{eq:ep:us:rk:fp:a}, the mechanics equation becomes a closed equation in $\Theta_u^{n,i}$, namely 
\begin{equation}
	\label{eq:us:closed}
	\big(I_s \otimes \cA + L\,I_s \otimes \widetilde{\cM}\big)\,\Theta_u^{n,i}
	= L\,(I_s \otimes \widetilde{\cM})\,\Theta_u^{n,i-1}
	- (I_s\otimes\cD^*)\,\cT^{-1}\,(I_s\otimes\cD)\,\Theta_u^{n,i-1}.
\end{equation}
Testing this equation with $(\diag(\RKb) \otimes \mathrm{Id})\,\Theta_u^{n,i}$ in the stage-tensor $\cV$-duality pairing $\langle\cdot,\cdot\rangle_\RKb$ and applying $\langle\widetilde{\cM}u, v\rangle = \langle\cC^{-1}\cD u, \cD v\rangle$ on the right-hand side, we obtain
\begin{equation}
	\label{eq:us:test:scalar}
	\Vert\Theta_u^{n,i}\Vert_{a,\RKb}^2
	+ L\,\big\vert\Theta_u^{n,i}\big\vert_{\widetilde{\cM},\RKb}^2
	= \big\langle \big[L\,(I_s\otimes\cC^{-1}) - \cT^{-1}\big]\,(I_s\otimes\cD)\,\Theta_u^{n,i-1},\ (I_s\otimes\cD)\,\Theta_u^{n,i}\big\rangle_\RKb.
\end{equation}

Since $\cC^{-1}\cB \colon \cHQ \to \cHQ$ is self-adjoint and positive on the $c$-inner product (by the ellipticity of $b$), the spectral theorem yields eigenpairs $(\nu_j, \tilde\phi_j)$ with
\begin{equation}
	\label{eq:us:spectral:B}
	\cB\,\tilde\phi_j = \nu_j\,\cC\,\tilde\phi_j,
	\qquad
	c(\tilde\phi_j, \tilde\phi_k) = \delta_{jk},
	\qquad
	\nu_j \geq c_b/C_c > 0.
\end{equation}
We expand $(I_s\otimes\cD)\Theta_u^{n,i-1}$ and $(I_s\otimes\cD)\Theta_u^{n,i}$ in the basis $\{\cC\tilde\phi_j\}$ of $\cQ^*$ stage-wise, leading to 
\[
	(I_s\otimes\cD)\Theta_u^{n,i-1} 
	= \sum_j \xi_j \otimes \cC\tilde\phi_j,
	\qquad
	(I_s\otimes\cD)\Theta_u^{n,i} 
	= \sum_j \eta_j \otimes \cC\tilde\phi_j
\]
with coefficient vectors $\xi_j, \eta_j \in \R^s$ whose components are $\xi_{j,\ell} = \langle\cD\Theta^{n,i-1}_{u,\ell}, \tilde\phi_j\rangle$ and $\eta_{j,\ell} = \langle\cD\Theta^{n,i}_{u,\ell}, \tilde\phi_j\rangle$. Applying $\cT^{-1}$ and $I_s\otimes\cC^{-1}$, respectively, we obtain by~\eqref{eq:us:spectral:B} 
\[
	\cT^{-1}\,(I_s\otimes\cD)\Theta_u^{n,i-1} 
	= \sum_j (I_s + \tau\nu_j\RKA)^{-1}\xi_j \otimes \tilde\phi_j,
	\quad
	(I_s\otimes\cC^{-1})(I_s\otimes\cD)\Theta_u^{n,i-1} 
	= \sum_j \xi_j \otimes \tilde\phi_j.
\]
%
The right-hand side of~\eqref{eq:us:test:scalar} therefore becomes the diagonal sum
\begin{equation}
	\label{eq:us:rhs:spectral}
	\big\langle \big[L\,(I_s\otimes\cC^{-1}) - \cT^{-1}\big]\,(I_s\otimes\cD)\Theta_u^{n,i-1},\ (I_s\otimes\cD)\Theta_u^{n,i}\big\rangle_\RKb
	= \sum_j \big\langle K_j \xi_j,\ \eta_j\big\rangle_\RKb,
\end{equation}
where
\begin{equation*}
	K_j
	\vcentcolon= L I_s - S_j \;\in\; \R^{s\times s},
	\qquad \text{with } S_j \vcentcolon= (I_s + \tau\nu_j\RKA)^{-1}.
\end{equation*}
We claim that 
\begin{equation}
	\label{eq:us:amp:bound}
	\Vert K_j\Vert_\RKb \leq \max(L,\,1-L) \qquad \text{for every } j,
\end{equation}
where $\Vert\cdot\Vert_\RKb$ denotes the matrix norm induced by the $\RKb$-weighted inner product on $\R^s$.
The assumed algebraic stability from  \Cref{ass:iterative:rk} implies $\mathrm{sym}(\diag(\RKb)\RKA) \succeq 0$ and, hence,
\begin{equation}
	\label{eq:us:accretive}
	x^\T\!\mathrm{sym}(\diag(\RKb)\,(I_s + \tau\nu_j\RKA))\,x
	= \Vert x\Vert_\RKb^2 + \tau\nu_j\,x^\T\!\mathrm{sym}(\diag(\RKb)\RKA)\,x
	\geq \Vert x\Vert_\RKb^2 
\end{equation}
for all $x\in\R^s$. 
Inserting $x = S_j y$ in~\eqref{eq:us:accretive} and using $(I_s + \tau\nu_j\RKA)\,S_j = I_s$ yields $\langle S_j y,\,y\rangle_\RKb \geq \Vert S_j y\Vert_\RKb^2$. 
An application of Cauchy--Schwarz then yields $\Vert S_j y\Vert_\RKb \leq \Vert y\Vert_\RKb$ for every $y\in\R^s$. Taking the supremum yields the contraction property of the resolvent, namely
\[
	\Vert S_j\Vert_\RKb \leq 1.
\]
An expansion of the squared $\RKb$-norm yields
\[
	\Vert(L I_s - S_j)\,y\Vert_\RKb^2
	= L^2 \Vert y\Vert_\RKb^2 - (2L-1)\,\Vert S_j y\Vert_\RKb^2
	- 2L\,\big(\langle y, S_j y\rangle_\RKb - \Vert S_j y\Vert_\RKb^2\big),
\]
where the last term is non-negative by~\eqref{eq:us:accretive}. 
Dropping it yields
\begin{equation}
	\label{eq:us:LI-S:bound}
	\Vert K_j y \Vert_\RKb^2
	= \Vert(L I_s - S_j)\,y\Vert_\RKb^2
	\leq L^2\Vert y\Vert_\RKb^2 - (2L-1)\,\Vert S_j y\Vert_\RKb^2.
\end{equation}
A case distinction on $L$ closes~\eqref{eq:us:amp:bound},  
\begin{itemize}[itemsep=0.2em]
	\item $L \geq \tfrac{1}{2}$: $(2L-1) \geq 0$ and~\eqref{eq:us:LI-S:bound} yields $\Vert K_j y \Vert_\RKb \leq L\, \Vert y\Vert_\RKb$;
	\item $L \leq \tfrac{1}{2}$: using $\Vert S_j y\Vert_\RKb \leq \Vert y\Vert_\RKb$ in~\eqref{eq:us:LI-S:bound} yields $\Vert K_j y\Vert_\RKb \leq (1-L)\Vert y\Vert_\RKb$.
\end{itemize}

By the $c$-orthonormality of $\{\tilde\phi_j\}$, Parseval's identity gives
\begin{equation*}
	\sum_j \Vert\xi_j\Vert_\RKb^2 = \big\vert\Theta_u^{n,i-1}\big\vert_{\widetilde{\cM},\RKb}^2,
	\qquad
	\sum_j \Vert\eta_j\Vert_\RKb^2 = \big\vert\Theta_u^{n,i}\big\vert_{\widetilde{\cM},\RKb}^2.
\end{equation*}
Applying Cauchy--Schwarz to each summand of~\eqref{eq:us:rhs:spectral} as well as to the sum over the modes, we conclude 
\begin{equation}
	\label{eq:us:rhs:bound}
	\bigg|\sum_j \langle K_j \xi_j,\,\eta_j\rangle_\RKb\bigg|
	\leq \max(L,\,1-L)\,\big\vert\Theta_u^{n,i-1}\big\vert_{\widetilde{\cM},\RKb}\,
	\big\vert\Theta_u^{n,i}\big\vert_{\widetilde{\cM},\RKb}.
\end{equation}
With the coupling strength $\omega$ from~\eqref{eqn:coupling:strength}, we have the Rayleigh inequality
\begin{equation*}
	\big\vert v\big\vert_{\widetilde{\cM}}^2
	= \big\Vert \cC^{-1/2}\cD\,v\big\Vert^2
	\leq \omega\,\Vert v\Vert_a^2. 
\end{equation*}
Weighted stage-wise by $\RKb$, this lifts to
\begin{equation}
	\label{eq:us:rayleigh:stage}
	\vert\Theta_u^{n,i}\vert_{\widetilde{\cM},\RKb}^2 \,\leq\, \omega\,\Vert\Theta_u^{n,i}\Vert_{a,\RKb}^2.
\end{equation}
Now, combining \eqref{eq:us:test:scalar}, \eqref{eq:us:rhs:spectral}, \eqref{eq:us:rhs:bound}, and~\eqref{eq:us:rayleigh:stage} 
gives
\[
	\frac{1+L\omega}{\omega}\,\big\vert\Theta_u^{n,i}\big\vert_{\widetilde{\cM},\RKb}
	\,\leq\, \max(L,\,1-L)\,\big\vert\Theta_u^{n,i-1}\big\vert_{\widetilde{\cM},\RKb},
\]
which is~\eqref{eq:rk:contraction:us} with rate $\rho_{\mathrm{US}}$ as defined in~\eqref{eq:rk:contraction:us:rate}.
The conditions for $\rho_{\mathrm{US}} < 1$ follow from a case analysis of
the stabilization parameter~$L$:
\begin{itemize}[itemsep=0.2em]
	\item $L \geq \tfrac{1}{2}$: $\rho_{\mathrm{US}} = L\omega/(1+L\omega) < 1$ unconditionally;
	\item $0 \leq L < \tfrac{1}{2}$: $\rho_{\mathrm{US}} = (1-L)\omega/(1+L\omega) < 1$ iff $\omega < 1/(1-2L)$.
\end{itemize}
The minimum is attained at $L = \tfrac{1}{2}$, yielding $\rho_{\mathrm{US}} = \omega/(2+\omega) < 1$.
\end{proof}

\begin{remark}[Case $L = 0$]
\label{rem:us:L:zero}
Without stabilization, i.e., for $L=0$, the contraction constant satisfies $\rho_{\mathrm{US}}\big|_{L=0} = \omega$. Hence, the iteration is contractive only for $\omega < 1$.
\end{remark}

\begin{remark}[Contraction in a genuine norm]
\label{rem:us:genuine:norm}
The form $\vert\cdot\vert_{\widetilde{\cM}}$ vanishes on $\ker(\cD)$, so $\vert\cdot\vert_{\widetilde{\cM},\RKb}$ is, in general, only a seminorm on $\cV^s$. One can show, however, that the iterates $\Theta_u^{n,i}$ live in a subspace on which $\vert\cdot\vert_{\widetilde{\cM},\RKb}$ is a norm equivalent to $\Vert\cdot\Vert_{a,\RKb}$, so that~\eqref{eq:rk:contraction:us} is indeed a genuine norm contraction.
\end{remark}

\subsection{Convergence of iterative \RK splittings}
\label{subsec:convergence:iterative:rk}

As shown in the previous subsections, the fixed-stress (\Cref{thm:rk:contraction:fs}) as well as the undrained-split
(\Cref{thm:rk:contraction:us}) iterations provide a contraction with rate
$\rho < 1$ in their respective norms if the stabilization parameter is chosen appropriately.
The convergence argument is identical for both splittings.
We write it generically using $\Vert\cdot\Vert_*$ acting on $\Theta^{n,i} \vcentcolon= X^{n,i} - X^{n,i-1}$ with
\begin{itemize}[itemsep=0.2em]
	\item $X = P$ and $\Vert\cdot\Vert_* = \Vert\cdot\Vert_{c,\RKb}$, $\rho = \rho_{\mathrm{FS}}$ for fixed-stress, 
	\item $X = U$ with $\Vert\cdot\Vert_* = \vert\cdot\vert_{\widetilde{\cM},\RKb}$, $\rho = \rho_{\mathrm{US}}$ for undrained-split.
\end{itemize}
Since $\RKb_\ell > 0$ for all~$\ell$ by \Cref{ass:iterative:rk},
the weighted norms $\Vert\cdot\Vert_{c,\RKb}$ and $\Vert\cdot\Vert_{a,\RKb}$ are equivalent
to the unweighted stage norms $\Vert\cdot\Vert_c$ and $\Vert\cdot\Vert_a$,
respectively. 
The same equivalence applies to $\vert\cdot\vert_{\widetilde{\cM},\RKb}$ versus $\vert\cdot\vert_{\widetilde{\cM},s}$.

\begin{theorem}[Convergence of iterative \RK splittings]
\label{thm:iter:convergence:rk}
Given \Cref{ass:iterative:rk,ass:smoothing},
let $u, p$ be the solutions of~\eqref{eq:ellpar:opt}
with sufficient temporal regularity,
and consider an $s$-stage \RK method of stage order~$q$ and classical order~$k$.
Let $(u^{n,J_n}, p^{n,J_n})$ denote the iterative solution
(fixed-stress or undrained-split)
after $J_n$ iterations satisfying the stopping criterion
\begin{equation}
	\label{eq:iter:stopping}
	\big\Vert \Theta^{n,J_n} \big\Vert_*
	\leq \tol.
\end{equation}
Then,
\begin{equation}
	\label{eq:conv:rk}
	\big\Vert u^{n,J_n} - u(t^n) \big\Vert^2_{\cV} + \big\Vert p^{n,J_n} - p(t^n) \big\Vert^2_{\cHQ}\,
	\lesssim\, \frac{\tol^2}{\tau^3} + \tau^{2\min(k,\,q+1+\smex)},
\end{equation}
where the hidden constant contains an exponential factor of the form $\eul^{Ct_n}$.
Setting $\tol = \tau^{\min(k,\,q+1+\smex)+3/2}$ balances both terms, yielding an overall error of order $\min(k,\,q+1+\smex)$.
\end{theorem}

\begin{proof}
Let $(u^n, p^n)$ denote the exact fully coupled \RK solution at time $t^n$
and $X^n$ the corresponding stage vector ($X = P$ for fixed-stress, $X = U$ for undrained-split), so that
\begin{equation}
	\label{eq:iter:err:decomp}
	\big\Vert p^{n,J_n} - p(t^n) \big\Vert_{\cHQ}
	\le \big\Vert p^{n,J_n} - p^n \big\Vert_{\cHQ}
	+ \big\Vert p^n - p(t^n) \big\Vert_{\cHQ}
\end{equation}
and analogously for $u$.
By the contraction results
(\Cref{thm:rk:contraction:fs,thm:rk:contraction:us}), 
we have after $J_n$ iterations
\[
	\Vert X^{n,J_n} - X^{n,J_n-1} \Vert_*
	\leq \rho^{J_n-1}\,
	\Vert X^{n,1} - X^{n,0} \Vert_*.
\]
Summing the geometric series, the stopping criterion~\eqref{eq:iter:stopping} yields $\Vert X^{n,J_n} - X^n \Vert_* \,\lesssim\, \tol$. 
Since $p^{n-1}$ (resp.\ $u^{n-1}$) is fixed across iterates and stiffly-accurate \RK methods satisfy
$p^n = \RKb^\T\RKA^{-1} P^n$ (resp.\ $u^n = \RKb^\T\RKA^{-1} U^n$)
by~\eqref{eq:rk:trans},
the time-step iteration error of the contracted component is bounded by the stage iteration error,
\begin{equation}
	\label{eq:iter:solution:bound}
	\big\Vert x^{n,J_n} - x^n \big\Vert
	\leq \|\RKb^\T\RKA^{-1}\|\,
	\Vert X^{n,J_n} - X^n \Vert_*
	\lesssim \tol,
\end{equation}
where $x \in \{p,u\}$ matches the contracted component $X$ and the LHS norm is $\Vert\cdot\Vert_{\cHQ}$ for fixed-stress and $\Vert\cdot\Vert_\cV$ for undrained-split.
The error of the other component is then controlled
by the elliptic coupling $\cA u = \cD^* p + f$ and the stability of $\cA^{-1}$.
For the discretization error of the fully coupled implicit \RK scheme
for~\eqref{eq:ellpar:opt}, we do not re-derive a Fourier stability estimate. Instead, we invoke the resolvent-smoothing analysis of~\cite[Thm.~3.3]{LubO95},
which, by~\Cref{ass:smoothing}, yields
\[
	\Vert p^n - p(t^n) \Vert_{\cHQ}
	\lesssim \tau^{\min(k,\,q+1+\smex)}.
\]
The per-step iteration error~\eqref{eq:iter:solution:bound} of order $\tol$
propagates through the parabolic structure,
since the stage derivatives involve a factor $1/\tau$ (from~\eqref{eq:rk:stage:deriv}). 
A discrete Gronwall argument over $N = T/\tau$ steps hence gives
\[
	\max_{1 \leq n \leq N} \big\Vert p^{n,J_n} - p^n \big\Vert_{\cHQ}^2
	\lesssim \frac{\tol^2}{\tau^3}.
\]
Combining this with the discretization bound via~\eqref{eq:iter:err:decomp} yields
\[
	\big\Vert p^{n,J_n} - p(t^n) \big\Vert_{\cHQ}^2
	\lesssim \frac{\tol^2}{\tau^3} + \tau^{2\min(k,\,q+1+\smex)}.
\]
The choice $\tol = \tau^{\min(k,\,q+1+\smex) + 3/2}$ balances both terms, yielding the overall error $\cO(\tau^{2\min(k,\,q+1+\smex)})$.
The bound for $u$ follows from the elliptic equation $\cA u = \cD^* p + f$
and the stability of $\cA^{-1}$.
\end{proof}
To summarize, both iterative approaches reach rate $\rho = \omega/(2 + \omega)$ if the optimal stabilization parameter ($L = \omega/2$ for fixed-stress, $L = 1/2$ for undrained-split) is chosen. Hence, both methods converge rapidly for small $\omega$.

\section{Numerical Experiments}
\label{sec:numerics}

We verify the theoretical convergence rates obtained in \Cref{sec:semiexplicit,sec:iterative}
using a manufactured solution for linear poroelasticity
on the unit square~$\Omega = (0,1)^2$ with $T = 1$.
We prescribe the exact solution as
\begin{align}\label{eq:mfg:solution}
	u(t,x,y)
	&= -\eul^{-At}
	\begin{bmatrix}\sin(\pi x)\sin(\pi y)\\\sin(\pi x)\sin(\pi y)\end{bmatrix},
	&
	p(t,x,y)
	&= \eul^{-At}\,\sin(\pi x)\sin(\pi y).
\end{align}
The forcing terms are chosen accordingly. 
The decay rate
\begin{equation}\label{eq:decay:rate}
	A = \frac{2\pi^2\,\kappa/\nu}{\alpha + 1/M}
\end{equation}
depends on all material parameters of the Biot system, which are chosen as 

\begin{center}
	\begin{tabular}{@{}c@{\qquad\quad}c@{\qquad\quad}c@{\qquad\quad}c@{\qquad\quad}c@{}}
		$\lambda$ & $\mu$ & $\tfrac{\kappa}{\nu}$ & $M$ & $\alpha$\\
		\midrule
		$1$ & $0.5$ & $0.1$ & $1$ & $0.1$\\
	\end{tabular}
\end{center}

The coupling strength~$\omega$ entering the weak coupling
condition~\eqref{eqn:weakCoupling}
is estimated at the continuous level by
$\omega \approx \alpha^2 M / (2\mu) = 10^{-2}$,
which is well below $1/(2^5-1) = 1/31$,
the bound for delay order $k = 5$
(Radau~IIA-3, the highest-order method tested).
With these parameters, the decay rate~\eqref{eq:decay:rate} evaluates to
$A \approx 1.79$, giving $\eul^{-AT} \approx 0.17$ at $T = 1$.

Spatial discretization is performed using Taylor--Hood finite elements for~$(u, p)$
on a uniform triangular mesh with $h = 2^{-6}$ ($64\times 64$ cells). To ensure that temporal errors dominate over spatial discretization effects across the tested time-step range, the polynomial degree is increased with the RK order:
\begin{align*}
	\text{Radau~IIA-1}&: (P_4,P_3), &
	\text{Radau~IIA-2}&: (P_6, P_5), &
	\text{Radau~IIA-3}&: (P_7, P_6).
\end{align*}
Each method is paired with $k = 2s{-}1$ delays matching its classical order.
Recall that Radau~IIA methods are stiffly accurate, which ensures
full classical order convergence for both the algebraic
variable~$u$ and the differential variable~$p$; cf.~\cite[Ch.~VI, Thm.~1.2]{HaiW96}.
All errors are reported in the $L^\infty(0,T)$ norm,
i.e., the maximum error over all time steps.
The convergence results are summarized in \Cref{fig:conv:semiexpl,fig:conv:iter}.

Both figures include gray dotted reference lines of slopes $1$, $3$, and $5$, corresponding to the classical orders of the Radau IIA family. For benchmarking, the monolithic implicit \RK errors (dashed red) are plotted in both figures.
For the implicit scheme, Radau IIA-$1$ and IIA-$2$ attain their classical orders $1$ and $3$ in both the $\cV$- and $\cHQ$-norms. Radau IIA-$3$ attains a rate of $\approx 4.2$ for the pressure in agreement with the bound $r = \min(k,\,q+1+\smex)$ of \Cref{thm:convergence}. For the displacement, this rate can be seen only for coarse time steps before it saturates for $\tau \leq 2^{-6}$ at the spatial discretization floor.
The semi-explicit pressure errors (\Cref{fig:conv:semiexpl}) essentially coincide with the implicit ones for all three methods. The semi-explicit displacement errors, however, lie clearly above the implicit ones across all tested $\tau$, while sharing the same asymptotic rate. 
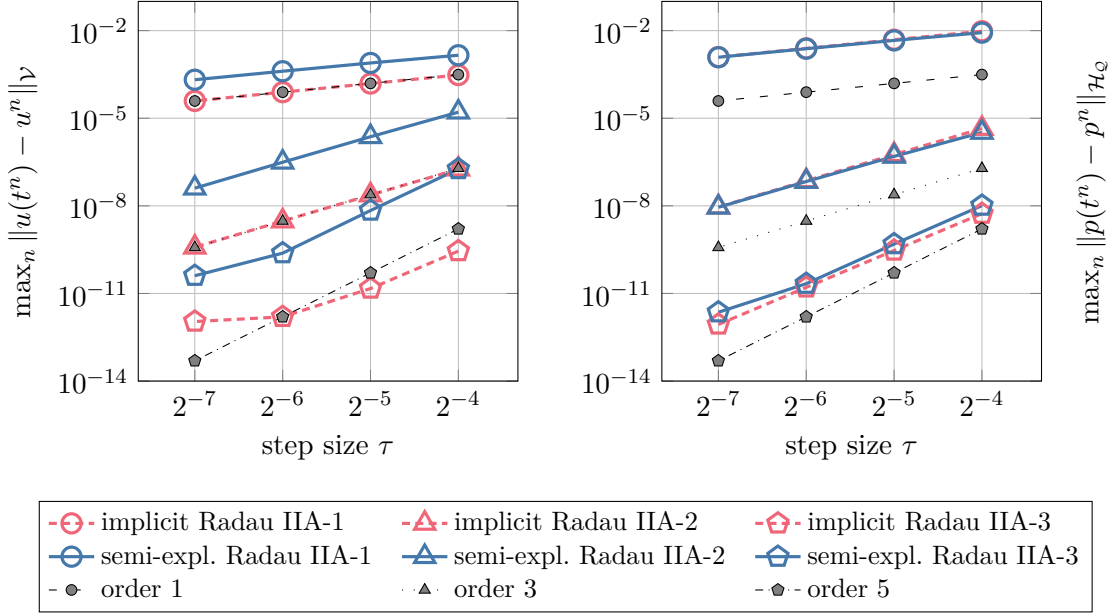
\begin{figure}
	\centering
	\begin{subfigure}[t]{.495\linewidth}
		\begin{tikzpicture}
  \begin{loglogaxis}[
    width=2.6in,
    height=2.6in,
    xmin=5.0e-03, xmax=1.0e-01,
    ymin=1.0e-14, ymax=1.0e-01,
    xtick={0.0625, 0.03125, 0.015625, 0.0078125},
    xticklabels={$2^{-4}$,$2^{-5}$,$2^{-6}$,$2^{-7}$},
    xlabel={step size $\tau$},
    ylabel={$\max_n \Vert u(t^n)-u^n \Vert_\cV$},
    xmajorgrids,
    ymajorgrids,
    cycle list name=rkconv,
    ]
    \addplot+ table[x={tau}, y={EPiRadauIIA1}] {data/conv_EPs_u_H1.dat};
    \addplot+ table[x={tau}, y={EPiRadauIIA2}] {data/conv_EPs_u_H1.dat};
    \addplot+ table[x={tau}, y={EPiRadauIIA3}] {data/conv_EPs_u_H1.dat};
    \addplot+ table[x={tau}, y={EPsRadauIIA1}] {data/conv_EPs_u_H1.dat};
    \addplot+ table[x={tau}, y={EPsRadauIIA2}] {data/conv_EPs_u_H1.dat};
    \addplot+ table[x={tau}, y={EPsRadauIIA3}] {data/conv_EPs_u_H1.dat};
    \addplot table[x={tau}, y={ORDER1}] {data/conv_EPs_u_H1.dat};
    \addplot table[x={tau}, y={ORDER3}] {data/conv_EPs_u_H1.dat};
    \addplot table[x={tau}, y={ORDER5}] {data/conv_EPs_u_H1.dat};
  \end{loglogaxis}
\end{tikzpicture}
	\end{subfigure}\hfill
	\begin{subfigure}[t]{.495\linewidth}
		\begin{tikzpicture}
  \begin{loglogaxis}[
    width=2.6in,
    height=2.6in,
    xmin=5.0e-03, xmax=1.0e-01,
    ymin=1.0e-14, ymax=1.0e-01,
    xtick={0.0625, 0.03125, 0.015625, 0.0078125},
    xticklabels={$2^{-4}$,$2^{-5}$,$2^{-6}$,$2^{-7}$},
    xlabel={step size $\tau$},
    ylabel={$\max_n \Vert p(t^n)-p^n \Vert_{\cHQ}$},
    ylabel style={at={(1.20,0.5)}},
    xmajorgrids,
    ymajorgrids,
    legend cell align={left},
    legend style={
      at={(0.5,-0.1)},
      anchor=north,
      /tikz/every even column/.append style={column sep=0.3cm},
      font=\small,
    },
    legend columns=3,
    legend to name=leg:semiexpl,
    cycle list name=rkconv,
    ]
  \addplot+ table[x={tau}, y={EPiRadauIIA1}] {data/conv_EPs_p_L2.dat};
  \addlegendentry{implicit Radau~IIA-$1$}
  \addplot+ table[x={tau}, y={EPiRadauIIA2}] {data/conv_EPs_p_L2.dat};
  \addlegendentry{implicit Radau~IIA-$2$}
  \addplot+ table[x={tau}, y={EPiRadauIIA3}] {data/conv_EPs_p_L2.dat};
  \addlegendentry{implicit Radau~IIA-$3$}
  \addplot+ table[x={tau}, y={EPsRadauIIA1}] {data/conv_EPs_p_L2.dat};
  \addlegendentry{semi-expl.\ Radau~IIA-$1$}
  \addplot+ table[x={tau}, y={EPsRadauIIA2}] {data/conv_EPs_p_L2.dat};
  \addlegendentry{semi-expl.\ Radau~IIA-$2$}
  \addplot+ table[x={tau}, y={EPsRadauIIA3}] {data/conv_EPs_p_L2.dat};
  \addlegendentry{semi-expl.\ Radau~IIA-$3$}
  \addplot table[x={tau}, y={ORDER1}] {data/conv_EPs_p_L2.dat};
  \addlegendentry{order $1$}
  \addplot table[x={tau}, y={ORDER3}] {data/conv_EPs_p_L2.dat};
  \addlegendentry{order $3$}
  \addplot table[x={tau}, y={ORDER5}] {data/conv_EPs_p_L2.dat};
  \addlegendentry{order $5$}
\end{loglogaxis}
\end{tikzpicture}
	\end{subfigure}\\[1em]
	\ref*{leg:semiexpl}
	\caption{Comparison of implicit (dashed red) and semi-explicit (solid blue) \RK discretization. 
	}
	\label{fig:conv:semiexpl}
\end{figure}

For the iterative schemes (\Cref{fig:conv:iter}), with optimal stabilizations $L=\omega/2$ (fixed-stress, cf.~\Cref{thm:rk:contraction:fs}) and $L=1/2$ (undrained-split, cf.~\Cref{thm:rk:contraction:us}), the converged pressure and displacement errors essentially coincide with the monolithic ones, confirming convergence to the fully coupled solution.
For Radau IIA-$3$, the iterative $u$-error plateaus instead of tracking the monolithic errors.

Finally, the average iteration counts presented in \Cref{tab:inner:iter:counts} only show a mild growth with decreasing $\tau$, using $\tol = \tau^{k+3/2}$.
\begin{table}[ht]
	\centering
	\caption{Average number of inner iterations per time step
	for the fixed-stress ($L=\omega/2$) and
	undrained-split ($L=1/2$) schemes,
	with stopping criterion~\eqref{eq:iter:stopping}
	and $\tol = \tau^{k+3/2}$.}
	\label{tab:inner:iter:counts}
	\begin{tabular}{lcccc}
		\toprule
		Method & $\tau = 2^{-4}$ & $\tau = 2^{-5}$ & $\tau = 2^{-6}$ & $\tau = 2^{-7}$ \\
		\midrule
		fixed-stress,  Radau~IIA-$1$ & 2.44 & 2.69 & 2.92 & 3.00 \\
		fixed-stress,  Radau~IIA-$2$ & 3.00 & 3.00 & 4.00 & 4.00 \\
		fixed-stress,  Radau~IIA-$3$ & 4.00 & 4.00 & 5.00 & 6.00 \\
		\midrule
		undrained-split, Radau~IIA-$1$ & 2.38 & 2.97 & 3.00 & 3.00 \\
		undrained-split, Radau~IIA-$2$ & 3.00 & 3.00 & 4.00 & 4.00 \\
		undrained-split, Radau~IIA-$3$ & 4.00 & 4.00 & 5.00 & 6.00 \\
		\bottomrule
	\end{tabular}
\end{table}

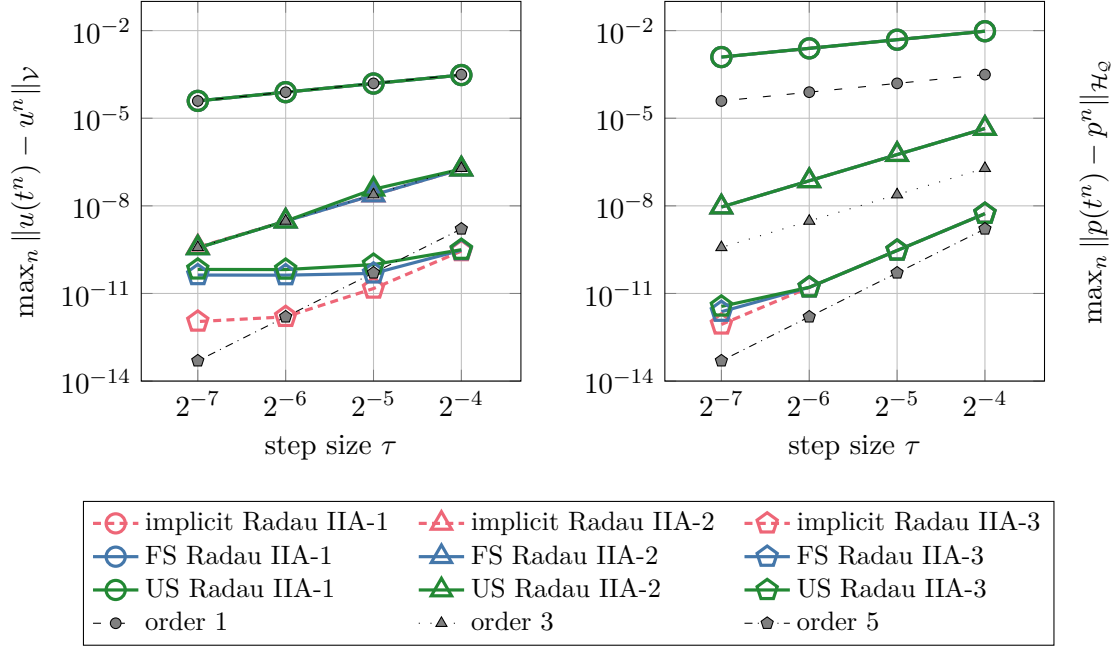
\begin{figure}
	\centering
	\begin{subfigure}[t]{.495\linewidth}
		\begin{tikzpicture}
  \begin{loglogaxis}[
    width=2.6in,
    height=2.6in,
    xmin=5.0e-03, xmax=1.0e-01,
    ymin=1.0e-14, ymax=1.0e-01,
    xtick={0.0625, 0.03125, 0.015625, 0.0078125},
    xticklabels={$2^{-4}$,$2^{-5}$,$2^{-6}$,$2^{-7}$},
    xlabel={step size $\tau$},
    ylabel={$\max_n \Vert u(t^n)-u^n \Vert_\cV$},
    xmajorgrids,
    ymajorgrids,
    cycle list name=rkconvIter,
    ]
    \addplot+ table[x={tau}, y={EPiRadauIIA1}] {data/conv_EPf_u_H1.dat};
    \addplot+ table[x={tau}, y={EPiRadauIIA2}] {data/conv_EPf_u_H1.dat};
    \addplot+ table[x={tau}, y={EPiRadauIIA3}] {data/conv_EPf_u_H1.dat};
    \addplot+ table[x={tau}, y={EPfRadauIIA1}] {data/conv_EPf_u_H1.dat};
    \addplot+ table[x={tau}, y={EPfRadauIIA2}] {data/conv_EPf_u_H1.dat};
    \addplot+ table[x={tau}, y={EPfRadauIIA3}] {data/conv_EPf_u_H1.dat};
    \addplot+ table[x={tau}, y={EPuRadauIIA1}] {data/conv_EPu_u_H1.dat};
    \addplot+ table[x={tau}, y={EPuRadauIIA2}] {data/conv_EPu_u_H1.dat};
    \addplot+ table[x={tau}, y={EPuRadauIIA3}] {data/conv_EPu_u_H1.dat};
    \addplot table[x={tau}, y={ORDER1}] {data/conv_EPf_u_H1.dat};
    \addplot table[x={tau}, y={ORDER3}] {data/conv_EPf_u_H1.dat};
    \addplot table[x={tau}, y={ORDER5}] {data/conv_EPf_u_H1.dat};
  \end{loglogaxis}
\end{tikzpicture}
	\end{subfigure}\hfill
	\begin{subfigure}[t]{.495\linewidth}
		\begin{tikzpicture}
  \begin{loglogaxis}[
    width=2.6in,
    height=2.6in,
    xmin=5.0e-03, xmax=1.0e-01,
    ymin=1.0e-14, ymax=1.0e-01,
    xtick={0.0625, 0.03125, 0.015625, 0.0078125},
    xticklabels={$2^{-4}$,$2^{-5}$,$2^{-6}$,$2^{-7}$},
    xlabel={step size $\tau$},
    ylabel={$\max_n \Vert p(t^n)-p^n \Vert_{\cHQ}$},
    ylabel style={at={(1.20,0.5)}},
    xmajorgrids,
    ymajorgrids,
    legend cell align={left},
    legend style={
      at={(0.5,-0.1)},
      anchor=north,
      /tikz/every even column/.append style={column sep=0.3cm},
      font=\small,
    },
    legend columns=3,
    legend to name=leg:iter,
    cycle list name=rkconvIter,
    ]
  \addplot+ table[x={tau}, y={EPiRadauIIA1}] {data/conv_EPf_p_L2.dat};
  \addlegendentry{implicit Radau~IIA-$1$}
  \addplot+ table[x={tau}, y={EPiRadauIIA2}] {data/conv_EPf_p_L2.dat};
  \addlegendentry{implicit Radau~IIA-$2$}
  \addplot+ table[x={tau}, y={EPiRadauIIA3}] {data/conv_EPf_p_L2.dat};
  \addlegendentry{implicit Radau~IIA-$3$}
  \addplot+ table[x={tau}, y={EPfRadauIIA1}] {data/conv_EPf_p_L2.dat};
  \addlegendentry{FS Radau~IIA-$1$}
  \addplot+ table[x={tau}, y={EPfRadauIIA2}] {data/conv_EPf_p_L2.dat};
  \addlegendentry{FS Radau~IIA-$2$}
  \addplot+ table[x={tau}, y={EPfRadauIIA3}] {data/conv_EPf_p_L2.dat};
  \addlegendentry{FS Radau~IIA-$3$}
  \addplot+ table[x={tau}, y={EPuRadauIIA1}] {data/conv_EPu_p_L2.dat};
  \addlegendentry{US Radau~IIA-$1$}
  \addplot+ table[x={tau}, y={EPuRadauIIA2}] {data/conv_EPu_p_L2.dat};
  \addlegendentry{US Radau~IIA-$2$}
  \addplot+ table[x={tau}, y={EPuRadauIIA3}] {data/conv_EPu_p_L2.dat};
  \addlegendentry{US Radau~IIA-$3$}
  \addplot table[x={tau}, y={ORDER1}] {data/conv_EPf_p_L2.dat};
  \addlegendentry{order $1$}
  \addplot table[x={tau}, y={ORDER3}] {data/conv_EPf_p_L2.dat};
  \addlegendentry{order $3$}
  \addplot table[x={tau}, y={ORDER5}] {data/conv_EPf_p_L2.dat};
  \addlegendentry{order $5$}
\end{loglogaxis}
\end{tikzpicture}
	\end{subfigure}\\[1em]
	\ref*{leg:iter}
	\caption{Comparicon of iterative \RK splittings schemes (fixed-stress in blue, undrained-split in green) with the monolithic implicit \RK method (dashed red).  
}
	\label{fig:conv:iter}
\end{figure}

\section{Conclusions}
\label{sec:conclusions}

We have presented a convergence analysis for decoupling \RK schemes
applied to elliptic--parabolic problems.
For the semi-explicit schemes based on a delay approximation,
we adapted the Fourier stability framework of~\cite{LubO95}
and established convergence of order~$k$ under weak coupling conditions for each delay order $k$, matching the bounds obtained for \BDF methods in~\cite{AltMU26}.
The consistency of results across \BDF and \RK time integrators
suggests that the coupling bounds are sharp and intrinsic
to the delay approximation structure.
For the iterative schemes (fixed-stress and undrained-split),
we combined contraction analysis with \RK consistency estimates,
using a spectral decomposition of the Schur complement operator
to establish the contraction property.
Future directions include the extension to nonlinear problems
and the treatment of variable time steps.

\appendix

\section*{Acknowledgments}
This project is funded by the Deutsche Forschungsgemeinschaft
(DFG, German Research Foundation) -- 467107679. BU further acknowledges funding by the Deutsche Forschungsgemeinschaft (DFG, German Research Foundation) -- Project-ID 258734477 -- SFB 1173. AM and BU acknowledge support by the Stuttgart Center for Simulation Science (SimTech).

\bibliographystyle{alpha}
\bibliography{literature}

\end{document}